\documentclass[fleqn,preprint,3p,a4paper]{elsarticle}

\usepackage{amssymb}
\usepackage{amsmath}
\usepackage{amsthm}
\usepackage{dcolumn}
\usepackage{endnotes}
\usepackage{tabularx}
\usepackage[matrix,arrow]{xy}
\usepackage{wasysym}

\newtheorem{theorem}{Theorem}[section]
\newtheorem{proposition}[theorem]{Proposition}
\newtheorem{lemma}[theorem]{Lemma}
\newtheorem{corollary}[theorem]{Corollary}

\theoremstyle{definition}
\newtheorem{definition}[theorem]{Definition}
\newtheorem{example}[theorem]{Example}
\newtheorem{remark}[theorem]{Remark}

\newcommand{\ir}{{\mathsf{Irr}}}

\newcommand{\mn}{\mathbb N}

\newcommand{\cl}{{\rm cl}}
\newcommand{\ii}{{\rm int}}

\newcommand{\ua}{\mathord{\uparrow}}
\newcommand{\da}{\mathord{\downarrow}}

\newcommand{\mk}{\mathord{\mathsf{K}}}
\newcommand{\wdd}{\mathord{\mathsf{WD}}}

\journal{Topology and its Applications}

\begin{document}

\begin{frontmatter}



\title{Non-reflective categories of some kinds of weakly sober spaces\tnoteref{t1}}
\tnotetext[t1]{This research is supported by the National Natural Science Foundation of China (No. 12071199, 11661057), the National Natural Science Foundation of Jiangxi Province, China (No. 20192ACBL20045) and the Natural Science Foundation of PhD start-up of Nanchang Hangkong University (EA202007056).}

\author[mymainaddress]{Xiaoquan Xu\corref{mycorrespondingauthor}}
\cortext[mycorrespondingauthor]{Corresponding author}
\ead{xiqxu2002@163.com}
\author[mysecondaryaddress]{Xinpeng Wen}
\ead{wenxinpeng2009@163.com}
\address[mymainaddress]{School of Mathematics and Statistics, Minnan Normal University, Zhangzhou 363000, China}
\address[mysecondaryaddress]{College of Mathematics and Information, Nanchang Hangkong University, Nanchang 330063, China}

\begin{abstract}
In \cite{Erne2018}, Ern\'e weakened the concept of sobriety in order to extend the theory of sober spaces and locally hypercompact spaces to situations where directed joins were missing, and introduced and discussed three kinds of non-sober spaces: cut spaces, weakly sober spaces, and quasisober spaces. In \cite{XSXZ-2020}, three other kinds of non-sober spaces, namely $\mathsf{DC}$ space, $\mathsf{RD}$ space and $\mathsf{WD}$ space, were introduced and investigated. All these six kinds of spaces are strictly weaker than sober spaces. In this paper, it is shown that none of the category of all $\mathsf{DC}$ spaces, that of all $\mathsf{RD}$ spaces, that of all $\mathsf{WD}$ spaces, that of all quasisober spaces, that of all weakly spaces and that of all cut spaces is reflective in the category of all $T_0$ spaces with continuous mappings.
\end{abstract}





\end{frontmatter}



\section{Introduction}

In domain theory and non-Hausdorff topology, the sober spaces, well-filtered spaces and $d$-spaces form three of the most important classes (see [2-13, 15-31]). Let $\mathbf{Top}_0$ be the category of all $T_0$ spaces with continuous mappings and $\mathbf{Sob}$ the full subcategory of $\mathbf{Top}_0$ containing all sober spaces. Denote the category of all $d$-spaces with continuous mappings and $\mathbf{Top}_w$ the full subcategory of $\mathbf{Top}_d$ containing all well-filtered spaces. It is well-known that $\mathbf{Sob}$ is reflective in $\mathbf{Top}_0$ (see \cite{redbook,Jean-2013}). Using $d$-closures, Wyler \cite{Wyler} proved that $\mathbf{Top}_d$ is reflective in $\mathbf{Top}_0$ (see also \cite{Ershov-1999,Keimel-Lawson,XXQ1, ZL-2017}). Recently, it was proved that $\mathbf{Top}_w$ is also reflective in $\mathbf{Top}_0$ (see \cite{LLW,SXXZ,wu-xi-xu-zhao-19,XXQ1,XSXZ-2020}).

In recent years, a number of classes of strongly sober spaces, strong $d$-spaces, weakly sober spaces, weakly well-filtered spaces and weakly $d$-spaces have been introduced and extensively studied from various different perspectives (see \cite{Erne2018, LL, MLZ, WX-2020, xu2020, xuzhao-2020, ZJW-2019, ZF-2010, ZH-2015}). In particular, in \cite{Erne2018}, Ern\'e relaxed the concept of sobriety in order to extend the theory of sober spaces and locally hypercompact spaces to situations where directed joins were missing. To that aim, he replaced joins by cuts, and introduced three kinds of non-sober spaces: cut spaces, weakly sober spaces, and quasisober spaces. This approach generalized and facilitated many results in the theory of quasicontinuous posets. For example, Ern\'e \cite{Erne2018} proved that the locally hypercompact weakly sober spaces are exactly the weak Scott spaces
of $s_2$-quasicontinuous posets in the sense of Zhang and Xu \cite{ZX-2015}. Some basic properties of cut spaces, weakly sober spaces and quasisober spaces were investigated in \cite{WX-2020}. In \cite{XSXZ-2020}, Xu, Shen, Xi and Zhao defined three new types of topological spaces --- directed closure spaces ($\mathsf{DC}$ spaces for short), Rudin spaces ($\mathsf{RD}$ spaces for short) and well-filtered determined spaces ($\mathsf{WD}$ spaces for short). The class of $\mathsf{RD}$ spaces lie between the class of $\mathsf{WD}$ spaces and that of $\mathsf{DC}$ spaces, while the class of $\mathsf{DC}$ spaces lies between the class of  $\mathsf{RD}$ spaces and that of sober spaces. It was shown by Ern\'e \cite[Proposition 3.2]{Erne2018} that every locally hypercompact $T_0$ space is a $\mathsf{DC}$ space. In \cite[Theorem 6.10 and Theorem 6.15]{XSXZ-2020} it was proved that every locally compact (resp., core-compact) $T_0$ space is a $\mathsf{RD}$ (resp., $\mathsf{WD}$) space. It was also shown in \cite[Theorem 5.6 and Theorem 6.12]{XSXZ-2021} that every $T_0$ space with a first-countable sobrification is an $\mathsf{RD}$ space and every first-countable $T_0$ space is a $\mathsf{WD}$ space.

Let $\mathbf{Top}_{dc}$ (resp., $\mathbf{Top}_{rd}$, $\mathbf{Top}_{wd}$, $\mathbf{QSob}$, $\mathbf{WSob}$, $\mathbf{Top}_{cut}$) be the category of all $\mathsf{DC}$ (resp., $\mathsf{RD}$, $\mathsf{WD}$, quasisober, weakly sober, cut) spaces with continuous mappings. Now a question naturally arises:

$\mathbf{Question.}$ Is $\mathbf{Top}_{dc}$ (resp., $\mathbf{Top}_{rd}$, $\mathbf{Top}_{wd}$, $\mathbf{QSob}$, $\mathbf{WSob}$, $\mathbf{Top}_{cut}$) reflective in $\mathbf{Top}_0$?

In this paper, we will answer the above question in the negative by showing that none of $\mathbf{Top}_{dc}$, $\mathbf{Top}_{rd}$, $\mathbf{Top}_{wd}$, $\mathbf{QSob}$, $\mathbf{WSob}$ and $\mathbf{Top}_{cut}$ is reflective in $\mathbf{Top}_0$.

\section{Preliminaries}

First we briefly recall some fundamental concepts and notations that will be used in the paper. For further details, we refer the reader to \cite{Engelking-1989, redbook, Jean-2013, MacLane-1997}.

For a subset $A$ of a poset $P$, let $\mathord{\downarrow}A=\{x\in P: x\leq  a \mbox{ for some } a\in A\}$ and $\mathord{\uparrow}A=\{x\in P: x\geq  a \mbox{ for some } a\in A\}$. For  $x\in P$, we write
$\mathord{\downarrow}x$ for $\mathord{\downarrow}\{x\}$ and $\mathord{\uparrow}x$ for $\mathord{\uparrow}\{x\}$. A subset $A$ is called a \emph{lower set} (resp., an \emph{upper set}) if $A=\downarrow A$ (resp., $A=\uparrow A$). Let $P^{(<\omega)}=\{F\subseteq P : F \mbox{~is a finite set}\}$ and $\mathbf{Fin} P=\{\uparrow F : F\in P^{(<\omega)}\setminus\{\emptyset\}\}$. A nonempty subset $D$ of a poset $P$ is called \emph{directed} if every two elements in $D$ have an upper bound in $D$. The set of all directed sets of $P$ is denoted by $\mathcal D(P)$. A subset $I\subseteq P$ is called an \emph{ideal} of $P$ if $I$ is a directed lower subset of $P$. The poset of all ideals (with the order of set inclusion) of $P$ is denoted by $\mathrm{Id}~\! P$. $P$ is called a \emph{directed complete poset}, or \emph{dcpo} for short, provided that $\vee D$ exists in $P$ for each $D\in \mathcal D(P)$. Obviously, $P$ is a dcpo iff $\vee I$ exists in $P$ for each $I\in \mathrm{Id}~\!P$. The poset $P$ is said to be \emph{Noetherian} if it satisfies the \emph{ascending chain condition} ($\mathrm{ACC}$ for short): every ascending chain has a greatest member. Clearly, $P$ is Noetherian if{}f every directed set of $P$ has a largest element (equivalently, every ideal of $P$ is principal).

As in \cite{redbook}, the \emph{upper topology} on a poset $Q$, generated
by the complements of the principal ideals of $Q$, is denoted by $\upsilon (Q)$; dually define the $lower$ $topology$ on $Q$ and denote it by $\omega  (Q)$. The upper sets of $Q$ form the (\emph{upper}) \emph{Alexandroff topology} $\alpha(Q)$. The space $(Q, \alpha(Q))$ is called the \emph{Alexandroff space} of $Q$. A subset $U$ of $Q$ is \emph{Scott open} if
(i) $U=\mathord{\uparrow}U$, and (ii) for any directed subset $D$ of $Q$ for which $\vee D$ exists, $\vee D\in U$ implies $D\cap U\neq\emptyset$. All Scott open subsets of $Q$ form a topology. This topology is called the \emph{Scott topology} on $Q$ and denoted by $\sigma(Q)$. The space $\Sigma~\!\! Q=(Q,\sigma(Q))$ is called the \emph{Scott space} of $Q$. For the chain $2=\{0, 1\}$ (with the order $0<1$), we have $\sigma(2)=\{\emptyset, \{1\}, \{0,1\}\}$. The space $\Sigma~\!\!2$
is well-known under the name of \emph{Sierpinski space}. The \emph{weak Scott topology} $\sigma_2 (Q)$ consists of all subsets $U$ of $Q$ such that $D\cap U=\emptyset$ implies $D^{\delta}\cap U=\emptyset$ for each $D\in \mathcal D(Q)$ (cf. \cite{Erne2018}). The space $\Sigma_2~\!\!Q=(Q, \sigma_2 (Q))$ is called the \emph{weak Scott space} of $Q$. In case $Q$ is a dcpo, the topology $\sigma_2 (Q)$ coincides with the usual Scott topology $\sigma (Q)$.

For a poset $P$ and $A\subseteq B\subseteq P$, $A^{\uparrow}$ and $A^{\downarrow}$ denote the sets of all upper and lower bounds of $A$ in $P$, respectively. Let $A^{\delta}=(A^{\uparrow})^{\downarrow}$. The set $A^{\delta}$ is called the \emph{cut closure} of $A$ in $P$. If $A^{\delta}=A$, we say that $A$ is a \emph{cut} in $P$. The complete lattice $\delta(P)=\{A^{\delta}: A\subseteq P\}$ (with the order of set inclusion) is called the \emph{Dedekind}-\emph{Macneille completion} of $P$ (cf. \cite{Erne-1991}).

The category of $T_0$ spaces and continuous mappings is denoted by $\mathbf{Top}_0$. For a $T_0$ space $X$ and $A\subseteq X$, the closure of $A$ (resp., the  interior of $A$) in $X$ is denoted by $\cl_X A$ (resp., $\ii_X A$) or simply by $\overline{A}$ (resp., $\ii A$) if there is no confusion. Let $X^{(<\omega)}=\{F\subseteq X : F \mbox{~is a finite set}\}$ and $X^{(\leqslant\omega)}=\{F\subseteq X : F \mbox{~is a countable set}\}$. We use $\leq_X$ to denote the \emph{specialization order} on $X$: $x\leq_X y$ if{}f $x\in \overline{\{y\}}$. The poset $X$ with the specialization order is denoted by $\Omega X$ or simply by $X$ if there is no confusion. Clearly, each open set of $X$ is an upper set and each closed set of $X$  is a lower set with respect to the specialization order. In the following, when a $T_0$ space $X$ is considered as a poset, the partial order always means the specialization order provided otherwise indicated. Let $\mathcal O(X)$ (resp., $\mathcal C(X)$) be the set of all open subsets (resp., closed subsets) of $X$. Define $\mathcal S_c(X)= \{\overline{\{x\}} : x\in X\}$ and $\mathcal D_c(X)=\{\overline{D} : D\in \mathcal D(X)\}$. For two spaces $X$ and $Y$, we use the symbol $X\cong Y$ to represent that $X$ and $Y$ are homeomorphic. For two posets $P$ and $Q$, we also use the symbol $P\cong Q$ to represent that $P$ and $Q$ are isomorphic. A {\it retract} of $X$ is a topological space $Z$ such that there are two continuous mappings $f:X\rightarrow Z$ and $g:Z\rightarrow X$ such that $f\circ g=id_{Z}$ (the identity mapping of $Z$ onto itself).

A $T_0$ space $X$ is called a $d$-space (or \emph{monotone convergence space}) if $X$ (with the specialization order) is a dcpo
 and $\mathcal O(X) \subseteq \sigma(X)$ (cf. \cite{redbook, Wyler}). Let $\mathbf{Top}_d$ demote the full subcategory of $\mathbf{Top}_0$ containing all $d$-spaces. A nonempty subset $A$ of a $T_0$ $X$ is said to be {\it irreducible} if for any $\{F_{1},F_{2}\}\subseteq \mathcal{C} (X)$, $A\subseteq F_{1}\cup F_{2}$ always implies $A\subseteq F_{1}$ or $A\subseteq F_{2}$. Denote by $\ir(X)$ (resp., $\ir_{c}(X)$) the set of all irreducible (resp., irreducible closed) subsets of $X$. The space $X$ is called \emph{sober}, if for any $A\in\ir_c(X)$, there is a unique point $x\in X$ such that $A=\overline{\{x\}}$. Let $\mathbf{Sob}$ be the full subcategory of $\mathbf{Top}_{0}$ containing all sober spaces.

For any $T_0$ space $X$, $\mathcal G\subseteq 2^{X}$ (the set of all subsets of $X$)) and $A\subseteq X$, let $\Diamond_{\mathcal G} A=\{G\in \mathcal G : G\cap A\neq\emptyset\}$ and $\Box_{\mathcal G} A=\{G\in \mathcal G : G\subseteq  A\}$. The symbols $\Diamond_{\mathcal G} A$ and $\Box_{\mathcal G} A$ will be simply written as $\Diamond A$ and $\Box A$ respectively if there is no confusion. The \emph{lower Vietoris topology} on $\mathcal{G}$ is the topology that has $\{\Diamond U : U\in \mathcal O(X)\}$ as a subbase, and the resulting space is denoted by $P_H(\mathcal{G})$. If $\mathcal{G}\subseteq \ir (X)$, then $\{\Diamond_{\mathcal{G}} U : U\in \mathcal O(X)\}$ is a topology on $\mathcal{G}$. The space $P_H(\mathcal{C}(X)\setminus \{\emptyset\})$ is called the \emph{Hoare power space} or \emph{lower space} of $X$ and is denoted by $P_H(X)$ for short (cf. \cite{Schalk}). Clearly, $P_H(X)=(\mathcal{C}(X)\setminus \{\emptyset\}, \upsilon(\mathcal{C}(X)\setminus \{\emptyset\}))$. So $P_H(X)$ is always sober (see, for example, \cite[Corollary 4.10]{ZH-2015} or \cite[Proposition 2.9]{XSXZ-2020}).

\begin{remark} \label{eta continuous}\label{remark:2.1} Let $X$ be a $T_0$ space.
\begin{enumerate}[\rm (1)]
	\item If $\mathcal{S}_c(X)\subseteq \mathcal{G}$, then the specialization order on $P_H(\mathcal{G})$ is the set inclusion order, and the \emph{canonical mapping} $\eta_{X}: X\longrightarrow P_H(\mathcal{G})$, given by $\eta_X(x)=\overline {\{x\}}$, is an order and topological embedding (cf. \cite{redbook, Jean-2013, Schalk}).
    \item The space $X^s=P_H(\ir_c(X))$ with the canonical mapping $\eta_{X}: X\longrightarrow X^s$ is a \emph{sobrification} of $X$ (cf. \cite{redbook, Jean-2013}). $\langle X^s=P_H(\ir_c(X)), \eta_{X}\rangle$ is called the \emph{canonical sobrification} of $X$.
\end{enumerate}
\end{remark}

A subset $A$ of a topological space $X$ is called \emph{saturated} if $A$ equals the intersection of all open sets containing it (equivalently, $A$ is an upper set in the specialization preorder). We shall use $\mathsf{K}(X)$ to denote the set of all nonempty compact saturated subsets of $X$ and endow it with the \emph{Smyth preorder}, that is, for $K_1,K_2\in \mathord{\mathsf{K}}(X)$, $K_1\sqsubseteq K_2$ if{}f $K_2\subseteq K_1$. The space $X$ is called \emph{well-filtered} if it is $T_0$, and for any filtered family $\mathcal{K}\subseteq \mathord{\mathsf{K}}(X)$ and any open set $U$, $\bigcap\mathcal{K}{\subseteq} U$ implies $K{\subseteq} U$ for some $K{\in}\mathcal{K}$. Let $\mathbf{Top}_w$ be the full subcategory of $\mathbf{Top}_{0}$ containing all well-filtered spaces.

We have the following implications (which can not be reversed):
\begin{center}
sobriety $\Rightarrow$ well-filteredness $\Rightarrow$ $d$-space.
\end{center}

For a $T_0$ space $X$, the space $P_S(\mathord{\mathsf{K}}(X))$, denoted shortly by $P_S(X)$, is called the \emph{Smyth power space} or \emph{upper space} of $X$ (cf. \cite{Heckmann, Schalk}). It is easy to verify that the specialization order on $P_S(X)$ is the Smyth order (that is, $\leq_{P_S(X)}=\sqsubseteq$). The \emph{canonical mapping} $\xi_X: X\longrightarrow P_S(X)$, $x\mapsto\ua x$, is an order and topological embedding (cf. \cite{Heckmann, Klause-Heckmann, Schalk}).

As in \cite{Erne2018}, a topological space $X$ is \emph{locally hypercompact} if for each $x\in X$ and each open neighborhood $U$ of $x$, there is  $\ua F\in X^{(<\omega)}$ such that $x\in\ii\,\ua F\subseteq\ua F\subseteq U$. The space $X$ is called \emph{core}-\emph{compact} if $(\mathcal O(X), \subseteq)$ is a \emph{continuous lattice} (cf. \cite{redbook}).

\begin{definition}\label{MD space} (\cite{Erne-2009})\label{definition:2.2}  A space $X$ is called \emph{monotone determined}
if any subset $U$ meeting all directed sets whose closure meets $U$ is open.
\end{definition}

\begin{definition}[\cite{Erne2018}]\label{definition:2.3} \rm Let $X$ be a $T_0$ space.
\begin{enumerate}[\rm (1)]
\item $X$ is said to be a \emph{cut space} if in $X$ (with the specialization order) any monotone net converges to each point in the cut closure of its range, or equivalently, $\cl_X D=D^{\delta}$ for each $D\in \mathcal D(X)$.

\item $X$ is called $weakly$ $sober$ if every irreducible closed set is a cut, that is, an intersection of point closures (with the whole space as intersection of the empty set).

\item $X$ is called $quasisober$ if each irreducible closed set is the cut closure of a directed set.
\end{enumerate}
\end{definition}

The category of all cut spaces (resp., all weakly sober spaces, all quasisober spaces) with continuous mappings is denoted by $\mathbf{Top}_{cut}$ (resp., $\mathbf{WSob}$, $\mathbf{QSob}$).

\begin{proposition} (\cite[Proposition 7]{Erne-2009} and \cite[Lemma 3.2]{Erne2018})\label{proposition:2.4}
\begin{enumerate}[\rm (1)]
\item The weak Scott topology is the coarsest ($weakest$) topology on a poset $P$ making it a monotone determined space with specialization poset $P$, the strongest one is the Alexandroff topology.
\item Every locally hypercompact space is monotone determined.

\item  The cut spaces are exactly those space whose topology is coarser than the weak Scott topology of the specialization poset.

\item The monotone determined cut spaces are exactly the weak Scott spaces of posets.
\end{enumerate}
\end{proposition}

\begin{lemma} (\cite[Lemma 5.1]{Erne2018})\label{lemma:2.5}
\begin{enumerate}[\rm (1)]
\item In a cut space, the cut closure of directed sets is irreducible and closed.

\item Every sober space is quasisober, every quasisober space is weakly sober, and every weakly sober space is a cut space.

\item A $T_0$ space is sober iff it is quasisober and directed complete in its specialization order.

\end{enumerate}
\end{lemma}

But a weakly sober space which is a dcpo in its specialization order is not always sober (see Example 2.6 below). Indeed, a $T_0$ space is sober iff it is weakly sober and every irreducible (closed) subset has a sup in its specialization order.

\begin{example} (\cite[Example 5.2, Example 5.3 and Example 5.4]{Erne2018})\label{example:2.6}
\begin{enumerate}[\rm (a)]
\item Adding top $\top$ and bottom $\bot$ to an infinite antichain yields a Noetherian lattice $L$ that is a $d$-space, hence a cut space, but not weakly sober in the topology $\upsilon(L)\cup \{\{\top\}\}$, because $L\backslash \{\top\}$ is irreducible and closed but not a cut.
\item The Scott space $\Sigma R$ of the real line $R$ is not sober $($as $R$ is not a dcpo$)$ but quasisober: the irreducible closed sets are the point closures and the whole space.
\item An infinite space $X$ with the cofinite topology is weakly sober but not quasisober: the irreducible closed sets are $X$ and the singletons. Whence $X$ is non-sober.
\end{enumerate}
\end{example}

For a $T_0$ space $X$ and $\mathcal{K}\subseteq \mathsf{K}(X)$, let $M(\mathcal{K})=\{A\in \mathcal C(X) : K\cap A\neq\emptyset \mbox{~for all~} K\in \mathcal{K}\}$ (that is, $\mathcal K\subseteq \Diamond A$) and $m(\mathcal{K})=\{A\in \mathcal C(X) : A \mbox{~is a minimal menber of~} M(\mathcal{K})\}$.

\begin{definition}\label{definition:2.7} (\cite[Definition 2.1]{SXXZ}, \cite[Definition 4.6]{XSXZ-2020})
		Let $X$ be a $T_0$ space. A nonempty subset  $A$  of $X$  is said to have \emph{Rudin property}, if there exists a filtered family $\mathcal K\subseteq \mathord{\mathsf{K}}(X)$ such that $\overline{A}\in m(\mathcal K)$ (that is, $\overline{A}$ is a minimal closed set that intersects all members of $\mathcal K$). Let $\mathsf{RD}(X)=\{A\in \mathcal C(X) : A\mbox{~has Rudin property}\}$. The sets in $\mathsf{RD}(X)$ will be called \emph{Rudin sets}.
\end{definition}

\begin{definition} (\cite[Definition 6.1]{XSXZ-2020})\label{definition:2.8}
 A subset $A$ of a $T_0$ space $X$ is called a \emph{well-filtered determined set}, $\mathsf{WD}$ \emph{set} for short, if
for any continuous mapping $f : X \rightarrow Y$ to a well-filtered space $Y$, there exists a unique $y_A \in Y$ such
that $f(A) = \overline{\{y_A \}}$. Denote by $\mathsf{WD}(X)$ the set of all closed well-filtered determined subsets of $X$.
\end{definition}

\begin{definition}(\cite[Definition 3.7, Definition 4.7 and Definition 6.7]{XSXZ-2020})\label{definition:2.9} Let $X$ be a $T_0$ space.
\begin{enumerate}[\rm (1)]
\item $X$ is called a \emph{directed closure space}, $\mathsf{DC}$ \emph{space} for short, if $\ir_c(X)=D_c(X)$, that is, for each $A\in \ir_c(X)$, there exists a directed subset $D$ of $X$ such that $A=\overline{D}$. The category of all $\mathsf{DC}$ spaces with continuous mappings is denoted by $\mathbf{Top}_{dc}$.

\item $X$ is called a \emph{Rudin space}, $\mathsf{RD}$ \emph{space} for short, if $\ir_c(X)=\mathsf{RD}(X)$, that is, every irreducible closed set of X is a Rudin set. The category of all Rudin spaces with continuous mappings is denoted by $\mathbf{Top}_{rd}$.

\item $X$  is called a \emph{well-filtered determined space}, $\mathsf{WD}$ \emph{space} for short, if all irreducible closed subsets of X are well-filtered
determined, that is, $\ir_c(X)=\mathsf{WD}(X)$. The category of all $\mathsf{WD}$ spaces with continuous mappings is denoted by $\mathbf{Top}_{wd}$.
\end{enumerate}
\end{definition}

\begin{lemma}\label{DRWIsetrelation} (\cite[Proposition 6.2]{XSXZ-2020})
	Let $X$ be a $T_0$ space. Then $\mathcal{S}_c(X)\subseteq\mathcal{D}_c(X)\subseteq \mathsf{RD}(X)\subseteq\mathsf{WD}(X)\subseteq\ir_c(X)$.
\end{lemma}

\begin{corollary}\label{SDRWspacerelation}(\cite[Corollary 6.3]{XSXZ-2020})
	Sober $\Rightarrow$ $\mathsf{DC}$ $\Rightarrow$ $\mathsf{RD}$ $\Rightarrow$ $\mathsf{WD}$.
\end{corollary}

\begin{theorem}\label{soberequiv} (\cite[Theorem 6.6]{XSXZ-2020}) For a $T_0$ space $X$, the following conditions are equivalent:
	\begin{enumerate}[\rm (1)]
		\item $X$ is sober.
		\item $X$ is a $\mathsf{DC}$ $d$-space.
        \item $X$ is a well-filtered $\mathsf{DC}$ space.
		\item $X$ is a well-filtered Rudin space.
		\item $X$ is a well-filtered $\mathsf{WD}$ space.
	\end{enumerate}
\end{theorem}

\section{Full adequate subcategories of $\mathbf{Top}_0$}

In this section, we recall some discussions and results about full adequate subcategories of $\mathbf{Top}_0$ in \cite{XXQ1} that will be used in the next section.

For a full subcategory $\mathbf{K}$ of $\mathbf{Top}_0$, the objects of $\mathbf{K}$ will be called $\mathbf{K}$-spaces. In \cite{Keimel-Lawson}, Keimel and Lawson proposed the following properties:

($\mathrm{K}_1$) Homeomorphic copies of $\mathbf{K}$-spaces are $\mathbf{K}$-spaces.

($\mathrm{K}_2$) All sober spaces are $\mathbf{K}$-spaces or, equivalently, $\mathbf{Sob}\subseteq \mathbf{K}$.

($\mathrm{K}_3$) In a sober space S, the intersection of any family of $\mathbf{K}$-subspaces is a $\mathbf{K}$-space.

($\mathrm{K}_4$) Continuous maps $f : S \longrightarrow T$ between sober spaces $S$ and $T$ are $\mathbf{K}$-continuous, that is, for every $\mathbf{K}$-subspace $K$ of $T$ , the inverse image $f^{-1}(K)$ is a $\mathbf{K}$-subspace of $S$.

\begin{definition}\label{K reflection} (\cite[Definition 4.1]{XXQ1})
Let $\mathbf{K}$ be a full subcategory of $\mathbf{Top}_0$ and $X$ a $T_{0}$ space. A $\mathbf{K}$-\emph{reflection} of $X$ is a pair $\langle \widetilde{X},\eta_X \rangle$ consisting of a $\mathbf{K}$-space $\widetilde{X}$ and a continuous mapping $\eta_X :X\rightarrow \widetilde{X}$ satisfying that for any continuous mapping $f:X\rightarrow Y$ to a $\mathbf{K}$-space, there exists a unique continuous mapping $f^{*}:\widetilde{X}\rightarrow Y$ such that $f^{*}\circ \eta_X =f$, that is, the following diagram commutes.

\begin{equation*}
\centerline{
\xymatrix{ X \ar[dr]_{f} \ar[r]^-{\eta_X}&  \widetilde{X}\ar@{.>}[d]^{f^{*}} & \\
  & Y  & &
   }}
\end{equation*}
\end{definition}

By a standard argument, $\mathbf{K}$-reflections, if they exist, are unique up to homeomorphism. We shall use $X^k$ to denote the space of the $\mathbf{K}$-reflection of $X$ if it exists.

The $\mathbf{Sob}$-reflections and $\mathbf{Top}_w$-reflections are exactly the sobrifications and well-filtered reflections respectively. The $\mathbf{Top}_d$-reflections are usually called the $d$-\emph{reflections} (see \cite{Ershov-1999, Keimel-Lawson, Wyler}). For simplicity, the well-filtered reflections (resp., $\mathbf{Top}_{dc}$-reflections, $\mathbf{Top}_{rd}$-reflections, $\mathbf{Top}_{wd}$-reflections, $\mathbf{Top}_{cut}$-reflections) will be called the $\mathbf{WF}$-\emph{reflections} (resp., $\mathbf{DC}$-\emph{reflections}, $\mathbf{RD}$-\emph{reflections}, $\mathbf{WD}$-\emph{reflections}, $\mathbf{Cut}$-\emph{reflections}).

In what follows in this section, $\mathbf{K}$ always refers to a full subcategory $\mathbf{Top}_0$ containing $\mathbf{Sob}$, that is, $\mathbf{K}$ has ($\mathrm{K}_2$). $\mathbf{K}$ is said to be \emph{closed with respect to homeomorphisms} if $\mathbf{K}$ has ($\mathrm{K}_1$). We call $\mathbf{K}$ a \emph{Keimel-Lawson category} if it satisfies ($\mathrm{K}_1$)-($\mathrm{K}_4$).

\begin{remark}\label{four categories closed with respect to homeomorphisms}  It is straightforward to verify that for $\mathbf{K}\in \{\mathbf{Sob}, \mathbf{Top}_w, \mathbf{Top}_d, \mathbf{Top}_{dc}, \mathbf{Top}_{rd}, \mathbf{Top}_{wd},\mathbf{QSob},$ $\mathbf{WSob}, \mathbf{Top}_{cut}\}$, $\mathbf{K}$ is closed with respect to homeomorphisms.
\end{remark}

\begin{definition}\label{K subset} (\cite[Definition 3.2]{XXQ1})
	 A subset $A$ of a $T_0$ space $X$ is called a $\mathbf{K}$-\emph{set}, provided for any continuous mapping $ f:X\longrightarrow Y$
to a $\mathbf{K}$-space $Y$, there exists a unique $y_A\in Y$ such that $\overline{f(A)}=\overline{\{y_A\}}$.
Denote by $\mathbf{K}(X)$ the set of all closed $\mathbf{K}$-sets of $X$.
\end{definition}

By the above definition, $\mathsf{WD}(X)=\mathbf{Top}_w(X)$. Clearly, a subset $A$ of a space $X$ is a $\mathbf{K}$-set if{}f $\overline{A}$ is a $\mathbf{K}$-set and. For simplicity, let $\mathbf{d}(X)=\mathbf{Top}_d(X)$.

\begin{lemma}\label{SKIsetrelation} (\cite[Lemma 3.3, Corollary 3.4 and Proposition 3.8]{XXQ1}) Let $X$ be a $T_0$ space. Then
\begin{enumerate}[\rm (1)]
\item $\mathbf{Sob}(X)=\ir_c(X)$.
\item $\mathcal S_c(X)\subseteq\mathbf{K}(X)\subseteq\ir_c(X)$.
\item  $\mathcal{D}_c(X)\subseteq \mathbf{d}(X)\subseteq\mathsf{WD}(X)$.
\end{enumerate}
\end{lemma}

\begin{corollary}\label{SKIsetrelation 1}  Let $X$ be a $T_0$ space and $A, B\in \mathcal C(X)$. Then the following two conditions are equivalent:
\begin{enumerate}[\rm (1)]
\item $A=B$.
\item $\Box_{\mathbf{K}(X)} A=\Box_{\mathbf{K}(X)} B$.
\end{enumerate}
\end{corollary}

\begin{proof}
	(1) $\Rightarrow$ (2): Trivial.

(2) $\Rightarrow$ (1): Suppose that $\Box_{\mathbf{K}(X)} A=\Box_{\mathbf{K}(X)} B$. Then for each $a\in A$ and $b\in B$, by Lemma \ref{SKIsetrelation} we have that $\overline{\{a\}}\in\Box_{\mathbf{K}(X)} A=\Box_{\mathbf{K}(X)} B$ and $\overline{\{b\}}\in\Box_{\mathbf{K}(X)} B=\Box_{\mathbf{K}(X)} A$, that is, $a\in B$ and $b\in A$. Whence $A\subseteq B$ and $B\subseteq A$. Thus $A=B$.
\end{proof}

By Lemma \ref{SKIsetrelation}, $\{\Diamond_{\mathbf{K}(X)} U : U\in \mathcal O(X)\}$ is a topology on $\mathbf{K}(X)$. In the following, let $\eta_X : X\longrightarrow P_H(\mathbf{K}(X))$, $\eta_X(x)=\overline {\{x\}}$, be the canonical mapping from $X$ to $P_H(\mathbf{K}(X))$. When $X=\Sigma~\!\!P$ for some poset $P$, $\eta_X$ is simply denoted by $\eta_P$.

\begin{lemma}\label{eta closure} (\cite[Lemma 4.3]{XXQ1})
	For a $T_0$ space $X$ be and $A\subseteq X$, $\overline{\eta_X(A)}=\overline{\eta_X\left(\overline{A}\right)}=\overline{\Box A}=\Box \overline{A}$ in $P_H(\mathbf{K}(X))$.
\end{lemma}

For the $\mathbf{K}$-reflections of $T_0$ spaces, the following lemma is crucial.

\begin{lemma}\label{K-lemmafstar} (\cite[Lemma 4.5]{XXQ1})
Let $X$ be a $T_0$ space and $f:X\longrightarrow Y$ a continuous mapping from $X$ to a $\mathbf{K}$-space $Y$. Then

\begin{enumerate}[\rm (1)]
\item the canonical mapping $\eta_X:X\longrightarrow P_H(\mathbf{K}(X))$ is a topological embedding, and

\item there exists a unique continuous mapping $f^* :P_H(\mathbf{K}(X))\longrightarrow Y$ such that $f^*\circ\eta_X=f$, that is, the following diagram commutes.
\begin{equation*}
\centerline{\xymatrix{
	X \ar[dr]_-{f} \ar[r]^-{\eta_X}
	&P_H(\mathbf{K}(X))\ar@{.>}[d]^-{f^*}\\
	&Y}}
\end{equation*}	
\noindent The unique continuous mapping $f^* :P_H(\mathbf{K}(X))\longrightarrow Y$ is defined by $f^*(A)=y_A$, where $y_A$ is the unique point of $Y$ such that $\overline{f(A)}=\overline{\{y_A\}}$.
\end{enumerate}
\end{lemma}

From Lemma \ref{K-lemmafstar} we immediately deduce the following result.

\begin{theorem}\label{K-reflection} (\cite[Theorem 4.6]{XXQ1})
	Let $X$ be a $T_0$ space. If $P_H(\mathbf{K}(X))$ is a $\mathbf{K}$-space, then the pair $\langle X^k=P_H(\mathbf{K}(X)), \eta_X\rangle$ is a  $\mathbf{K}$-reflection of $X$.
\end{theorem}

\begin{definition}\label{K-adequate}  (\cite[Definition 4.7]{XXQ1}) A full subcategory $\mathbf{K}$ of $\mathbf{Top}_0$ is called \emph{adequate} if for any $T_0$ space $X$, $P_H(\mathbf{K}(X))$ is a $\mathbf{K}$-space.
\end{definition}

\begin{theorem}\label{K-adequate reflective}  (\cite[Corollary 4.8]{XXQ1})
	If $\mathbf{K}$ is adequate, then $\mathbf{K}$ is reflective in $\mathbf{Top}_0$.
\end{theorem}

\begin{theorem}\label{four categories adequate}  (\cite[Proposition 5.1, Theorem 5.4, Theorem 5.14 and Theorem 5.17]{XXQ1}) $\mathbf{Sob}$,  $\mathbf{Top}_d$ and $\mathbf{Top}_w$ all are adequate. Moreover, every Keimel-Lawson category $\mathbf{K}$ is adequate. Therefore, they all are reflective in $\mathbf{Top}_0$.
\end{theorem}

The following lemma gives a characterization of $\mathbf{K}$-spaces by $\mathbf{K}$-sets.

\begin{lemma}\label{K-space charac by K-set} (\cite[Corollary 4.10]{XXQ1}) Let $\mathbf{K}$ be a full subcategory of $\mathbf{Top}_0$ containing $\mathbf{Sob}$ and $X$ a $T_0$ space. Suppose that $\mathbf{K}$ is adequate and closed with respect to homeomorphisms. Then the following two conditions are equivalent:
\begin{enumerate}[\rm (1)]
\item $X$ is a $\mathbf{K}$-space.
\item $\mathbf{K}(X)=\mathcal S_c(X)$.
\end{enumerate}
\end{lemma}

\begin{corollary}\label{two Ks equiv}  Let $\mathbf{K}_1, \mathbf{K}_2$ be two full subcategories of $\mathbf{Top}_0$ containing $\mathbf{Sob}$. Suppose that both $\mathbf{K}_1$ and $\mathbf{K}_2$ are adequate and closed with respect to homeomorphisms. Then the following two conditions are equivalent:
\begin{enumerate}[\rm (1)]
\item $\mathbf{K}_1=\mathbf{K}_2$.
\item For any $T_0$ space $X$, $\mathbf{K}_1(X)=\mathbf{K}_2(X)$.
\end{enumerate}
\end{corollary}

\begin{proof} (1) $\Rightarrow$ (2): Trivial.

(2) $\Rightarrow$ (1): Suppose that $X$ is a $\mathbf{K}_1$-space. Then by (2) and Lemma \ref{K-space charac by K-set}, $\mathbf{K}_2(X)=\mathbf{K}_1(X)=\mathcal S_c(X)$, whence by Lemma \ref{K-space charac by K-set} again, $X$ is a $\mathbf{K}_2$-space. So $\mathbf{K}_1\subseteq_{full} \mathbf{K}_2$. Similarly we have $\mathbf{K}_2\subseteq_{full} \mathbf{K}_1$. Thus $\mathbf{K}_1=\mathbf{K}_2$.

\end{proof}

One immediately deduce the following two corollaries from Remark \ref{four categories closed with respect to homeomorphisms}, Theorem \ref{four categories adequate} and Lemma \ref{K-space charac by K-set}.

\begin{corollary}\label{WF space charac by WD set} (\cite[Corollary 7.11]{XSXZ-2020})
For a $T_0$ space $X$, the following two conditions are equivalent:
	\begin{enumerate}[\rm (1)]
		\item $X$ is well-filtered.
        \item $\wdd (X)=\mathcal S_c(X)$.
	\end{enumerate}
\end{corollary}

\begin{corollary}\label{d space charac by d set}
For a $T_0$ space $X$, the following two conditions are equivalent:
	\begin{enumerate}[\rm (1)]
		\item $X$ is a $d$-space.
        \item $\mathbf{d} (X)=\mathcal S_c(X)$.
	\end{enumerate}
\end{corollary}

\begin{theorem}\label{K-reflections coincide}  Let $\mathbf{K}_1, \mathbf{K}_2$ be two full subcategories of $\mathbf{Top}_0$ containing $\mathbf{Sob}$ and $X$ a $T_0$ space. Suppose that both $\mathbf{K}_1$ and $\mathbf{K}_2$ are adequate and closed with respect to homeomorphisms. Then the following two conditions are equivalent:
\begin{enumerate}[\rm (1)]
\item $\mathbf{K}_1(X)=\mathbf{K}_2(X)$.
\item The $\mathbf{K}_1$-reflection of $X$ and the $\mathbf{K}_2$-reflection of $X$ coincide.

\end{enumerate}
\end{theorem}

\begin{proof} (1) $\Rightarrow$ (2): By Theorem \ref{K-reflection}.

(2) $\Rightarrow$ (1): By Theorem \ref{K-reflection}, $\langle P_H(\mathbf{K}_1(X)), \eta_X\rangle$ is a $\mathbf{K}$-reflection of $X$ and $\langle P_H(\mathbf{K}_2(X)), \hat{\eta}_X\rangle$ is a $\mathbf{K}_2$-reflection of $X$, where $\eta_X : X\longrightarrow P_H(\mathbf{K}_1(X))$ and $\hat{\eta}_X : X\longrightarrow P_H(\mathbf{K}_2(X))$ are defined by $\eta_X(x)=\overline {\{x\}}$ and $\hat{\eta}_X(x)=\overline {\{x\}}$ for each $x\in X$, respectively. Since the $\mathbf{K}_1$-reflection of $X$ agrees with the $\mathbf{K}_2$-reflection of $X$, there exists a unique continuous mapping $(\hat{\eta}_X)^{k}:P_H(\mathbf{K}_1(X))\rightarrow P_H(\mathbf{K}_2(X))$ such that $(\hat{\eta}_X)^{k}\circ \eta_X =\hat{\eta}_X$, and also there exists a unique continuous mapping $({\eta}_X)^{k}:P_H(\mathbf{K}_2(X))\rightarrow P_H(\mathbf{K}_1(X))$ such that $({\eta}_X)^{k}\circ \hat{\eta}_X =\eta_X$, that is, the following two diagrams commute.

\begin{equation*}
\centerline{
\xymatrix{ X \ar[dr]_{\hat{\eta}_X} \ar[r]^-{\eta_X}&  P_H(\mathbf{K}_1(X))\ar@{.>}[d]^{(\hat{\eta}_X)^{k}} &\\
  & P_H(\mathbf{K}_2(X))  & &
   }
 \xymatrix{ X \ar[dr]_{\hat{\eta}_X} \ar[r]^-{\eta_X}&  P_H(\mathbf{K}_1(X))\ar@{<.}[d]^{(\eta_X)^{k}} &\\
  & P_H(\mathbf{K}_2(X))  & &
   } }
\end{equation*}

By Lemma \ref{K-lemmafstar}, the unique continuous mapping $(\hat{\eta}_X)^{k}:P_H(\mathbf{K}_1(X))\rightarrow P_H(\mathbf{K}_2(X))$  is defined by $(\hat{\eta}_X)^{k}(A)=B_A$, where $B_A$ is the unique point of $\mathbf{K}_2(X)$ such that $\overline{\eta_X(A)}=\overline{\{B_A\}}=\Box_{\mathbf{K}_2(X)} B_A$ in $P_H(\mathbf{K}_2(X))$.
In the same way, the unique continuous mapping $(\eta_X)^{k}:P_H(\mathbf{K}_2(X))\rightarrow P_H(\mathbf{K}_1(X))$  is defined by $(\eta_X)^{k}(B)=A_B$, where $A_B$ is the unique point of $\mathbf{K}_1(X)$ such that $\overline{\hat{\eta}_X(B)}=\overline{\{A_B\}}=\Box_{\mathbf{K}_1(X)} A_B$ in $P_H(\mathbf{K}_1(X))$. For each $A\in \mathbf{K}_1(X)$ and $B\in \mathbf{K}_2(X)$, by Lemma \ref{eta closure} we have that $\Box_{\mathbf{K}_1(X)} B_A=\overline{\{B_A\}}=\overline{\eta_X(A)}=\Box_{\mathbf{K}_1(X)} A$ and  $\Box_{\mathbf{K}_2(X)} A_B=\overline{\{A_B\}}=\overline{\hat{\eta}_X(B)}=\Box_{\mathbf{K}_2(X)} B$, and consequently, $A=B_A\in \mathbf{K}_2(X)$ and $B=A_B\in\mathbf{K}_1(X)$ by Corollary \ref{SKIsetrelation 1}. Thus $\mathbf{K}_1(X)=\mathbf{K}_2(X)$ and both $(\hat{\eta}_X)^{k}$ and $(\eta_X)^{k}$ are the identity mapping.

\end{proof}

\begin{corollary}\label{K-reflections coincide 1}  Let $\mathbf{Sob}\subseteq_{full}\mathbf{K}_1\subseteq_{full}\mathbf{K}_2\subseteq_{full} \mathbf{Top}_0$ and $X$ a $T_0$ space. Suppose that both $\mathbf{K}_1$ and $\mathbf{K}_2$ are adequate and closed with respect to homeomorphisms. Then the following conditions are equivalent:
\begin{enumerate}[\rm (1)]
\item $\mathbf{K}_1(X)=\mathbf{K}_2(X)$.
\item $X^{k_1}=P_H(\mathbf{K}_1(X))$ is a $\mathbf{K}_2$-space.
\item $\langle X^{k_1}=P_H(\mathbf{K}_1(X)), \eta_X\rangle)$ is a $\mathbf{K}_2$-reflection of $X$, where $\eta_X : X\longrightarrow P_H(\mathbf{K}_1(X))$ is defined by $\eta_X(x)=\overline {\{x\}}$ for each $x\in X$.
    \item $\langle X^{k_2}=P_H(\mathbf{K}_2(X)), \hat{\eta}_X\rangle)$ is a $\mathbf{K}_1$-reflection of $X$, where $\hat{\eta}_X : X\longrightarrow P_H(\mathbf{K}_2(X))$ is defined by $\hat{\eta}_X(x)=\overline {\{x\}}$ for each $x\in X$.
\item The $\mathbf{K}_1$-reflection of $X$ and the $\mathbf{K}_2$-reflection of $X$ coincide.
\end{enumerate}
\end{corollary}

\begin{proof} (1) $\Rightarrow$ (2), (1) $\Rightarrow$ (3), (1) $\Rightarrow$ (4) and (3) $\Rightarrow$ (2): Trivial.

 (1) $\Leftrightarrow$ (5): By Theorem \ref{K-reflections coincide}.

(2) $\Rightarrow$ (1):  Since $\mathbf{K}_1$ is a full and adequate subcategory $\mathbf{Top}_0$ containing $\mathbf{Sob}$, $\langle X^{k_1}=P_H(\mathbf{K}_1(X)), \eta_X\rangle)$ is a $\mathbf{K}_1$-reflection of $X$, whence by (2) and $\mathbf{K}_1\subseteq_{full}\mathbf{K}_2$, $\langle X^{k_1}, \eta_X\rangle)$ is also a $\mathbf{K}_2$-reflection of $X$. Thus $\mathbf{K}_1(X)=\mathbf{K}_2(X)$ by Theorem \ref{K-reflections coincide}.

(4) $\Rightarrow$ (1):  Since $\mathbf{K}_2$ is a full and adequate subcategory $\mathbf{Top}_0$ containing $\mathbf{Sob}$, $\langle X^{k_2}=P_H(\mathbf{K}_2(X)), \hat{\eta}_X\rangle)$ is a $\mathbf{K}_2$-reflection of $X$. By (4), $\langle X^{k_2}=P_H(\mathbf{K}_2(X)), \hat{\eta}_X\rangle)$ is also a $\mathbf{K}_1$-reflection of $X$. Thus $\mathbf{K}_1(X)=\mathbf{K}_2(X)$ by Theorem \ref{K-reflections coincide}.

\end{proof}

By Remark \ref{four categories closed with respect to homeomorphisms}, Theorem \ref{K-adequate reflective}, Theorem \ref{four categories adequate} and Corollary \ref{K-reflections coincide 1}, we get the following three corollaries.

\begin{corollary}\label{WF reflection agrees with sobrification} (\cite[Theorem 7.16]{XSXZ-2020})
	For a $T_0$ space $X$, the following conditions are equivalent:
	\begin{enumerate}[\rm (1)]
		\item $X$ is a $\mathsf{WD}$ space, that is, $\ir_c(X)=\mathsf{WD}(X)$.
\item $X^w=P_H(\mathsf{WD}(X))$ is sober.
\item $\langle X^w=P_H(\mathsf{WD}(X)), \eta_X \rangle$ is a sobrification of $X$.
\item $\langle X^s=P_H(\ir_c(X)), \eta_X \rangle$ is a $\mathbf{WF}$-reflection of $X$.
\item The sobrification of $X$ and the $\mathbf{WF}$-reflection of $X$ coincide.
	\end{enumerate}
\end{corollary}

\begin{corollary}\label{d reflection agrees with WF reflection}
	For a $T_0$ space $X$, the following conditions are equivalent:
	\begin{enumerate}[\rm (1)]
		\item $\mathsf{WD}(X)=\mathbf{d}(X)$.
\item $X^d=P_H(\mathbf{d}(X))$ is well-filtered.
\item $\langle X^d=P_H(\mathbf{d}(X)), \eta_X \rangle$ is a $\mathbf{WF}$-reflection of $X$.
\item $\langle X^w=P_H(\mathsf{WD}(X)), \eta_X \rangle$ is a $d$-reflection of $X$.
\item The $\mathbf{WF}$-reflection of $X$ and the $d$-reflection of $X$ coincide.
	\end{enumerate}
\end{corollary}

\begin{corollary}\label{d reflection agrees with sobrification}
	For a $T_0$ space $X$, the following conditions are equivalent:
	\begin{enumerate}[\rm (1)]
		\item $\ir_c(X)=\mathbf{d}(X)$.
\item $X^d=P_H(\mathbf{d}(X))$ is sober.
\item $\langle X^d=P_H(\mathbf{d}(X)), \eta_X \rangle$ is a sobrification of $X$.
\item $\langle X^s=P_H(\ir_c(X)), \eta_X \rangle$ is a $d$-reflection of $X$.
\item The sobrification of $X$ and the $d$-reflection of $X$ coincide.
	\end{enumerate}
\end{corollary}

By Lemma \ref{DRWIsetrelation}, Lemma \ref{SKIsetrelation}, Corollary \ref{WF reflection agrees with sobrification}, Corollary \ref{d reflection agrees with WF reflection} and Corollary \ref{d reflection agrees with sobrification}, we get the following corollary.

\begin{corollary}\label{DC space implies d reflection agrees with sobrification}
	If $X$ is a $\mathsf{DC}$ space, then
\begin{enumerate}[\rm (1)]
		\item $\ir_c(X)=\mathsf{WD}(X)=\mathsf{RD}(X)=\mathbf{d}(X)$.
\item $X^d=P_H(\mathbf{d}(X))$ is sober.
\item The sobrification of $X$, the $\mathbf{WF}$-reflection and the $d$-reflection of $X$ all coincide.
\item $\langle X^s=P_H(\ir_c(X)), \eta_X \rangle$ is both a sobrification of $X$ and a $d$-reflection of $X$, and it is also a $\mathbf{WF}$-reflection of $X$.
\end{enumerate}
\end{corollary}

\section{Main results}

For a category $A$, we will use the symbol $B\subseteq_{full} A$ to denote that $B$ is a full subcategory of $A$. More generally, we give the following definition (cf. \cite[Section IV-3]{MacLane-1997})

\begin{definition} \label{reflection}
Let $A$ be a category and $B$ a subcategory $A$. For an object $a$ of $A$, a \emph{reflection of} $a$ \emph{in} $B$, $B$-\emph{reflection} of $a$ for short, is a pair $\langle b_a,\eta_a \rangle$ consisting of an object $b_a$ of $B$ and an arrow $\eta_a :a\rightarrow b_a$ satisfying that for any arrow $f: a\rightarrow b$ from $a$ to $b$ in $A$, there exists a unique arrow $f^{*}:b_a\rightarrow b$ in $B$ such that $f^{*}\circ \eta_a =f$, that is, the following diagram commutes.

\begin{equation*}
\centerline{
\xymatrix{ a \ar[dr]_{f} \ar[r]^-{\eta_a}&  b_a\ar@{.>}[d]^{f^{*}} & \\
  & b  & &
   }}
\end{equation*}
\end{definition}

\begin{definition}\label{X plus top 1} Let $X$ be a topological space for which $X$ is irreducible (i.e., $X\in\ir_c(X)$). Select a point $\top$ such that $\top\not\in X$. Then $(\mathcal C(X)\setminus \{X\})\bigcup \{X\cup\{\top\}\}$ (as the set of all closed sets) is a topology on $X\cup\{\top\}$. The resulting space is denoted by $X_{\top}^{\flat}$. Define a mapping $\eta_{X}^{\flat}: X\rightarrow  X_{\top}^{\flat}$ by  $\eta_{X}^{\flat}(x)=x$ for each $x\in X$. Clearly,  $\eta_{X}^{\flat}$ is a topological embedding.
\end{definition}

\begin{definition}\label{X plus top 2} Let $X$ be a $T_0$ space for which $X$ (with the specialization order) has a greatest element $\top_X$ and $X\setminus \{\top_X\}\in\ir_c(X)$. Select a point $\top$ such that $\top\not\in X$. Then $(\mathcal C(X)\setminus \{X\setminus \{\top_X\}\})\bigcup \{X\cup\{\top\}\}$ (as the set of all closed sets) is a topology on $X\cup\{\top\}$. The resulting space is denoted by $X_{\top}^{\natural}$. Define a mapping $\eta_{X}^{\natural}: X\rightarrow  X_{\top}^{\natural}$ by $\eta_{X}^{\natural}(x)=x$ for each $x\in X\setminus \{\top_X\}$ and $\eta_{X}^{\natural}(\top_X)=\top$. It is straightforward to verify that $\eta_{X}^{\natural}$ is a topological embedding.
\end{definition}

\begin{definition}\label{NK space} A $T_0$ space $X$ is called a $\mathbf{K}^{\neg}$-\emph{space} if it satisfies the following four conditions:
\begin{enumerate}[\rm (1)]
\item $X$ is not a $\mathbf{K}$-space.
\item $X$ (with the specialization order) has a greatest element $\top_X$.
 \item $X\setminus \{\top_X\}\in\ir_c(X)$.
  \item $\overline{\{x\}}\neq X\setminus \{\top_X\}$ for each $x\in X$, or equivalently, $X\setminus \{\top_X\}$ has no greatest element.
  \end{enumerate}
\end{definition}

Since $\overline{\{\top_X\}}=\da \top_X=X\neq X\setminus \{\top_X\}$, condition (4) in Definition \ref{NK space} is equivalent to the condition:  $\overline{\{x\}}\neq X\setminus \{\top_X\}$ for each $x\in X\setminus \{\top_X\}$.

\begin{remark}\label{X plus top flat} Let $X$ be a topological space and $X$ is not irreducible. Then there exist $A, B\in \mathcal C(X)\setminus \{X\}$ such that $X=A\cup B$, whence $(\mathcal C(X)\setminus \{X\})\bigcup \{X\cup\{\top\}\}$ is not a topology on $X\cup\{\top\}$.
\end{remark}

In order to give our main results, we need the following four results of \cite{SBWX-2021}.

\begin{lemma}\label{key lemma 1} (\cite[Lemma 6.3]{SBWX-2021})
Let $A$ be a category and $C\subseteq_{full} B\subseteq_{full} A$. For an object $a$ of $A$, if $\langle b_a, \eta_{a} : a \rightarrow b_a\rangle$ is a $B$-reflection of $a$ and $\langle c_a, \xi_a : a \rightarrow c_a\rangle $ is a $C$-reflection of $a$, then there is a unique arrow $\xi_{a}^{*} : b_a \rightarrow c_a$ in $B$ such that $\xi_{a}^{*}\circ \eta_a=\xi_a$, that is, the following diagram commutes.

\begin{equation*}
\centerline{
\xymatrix{ a \ar[dr]_{\xi_a} \ar[r]^-{\eta_a}&  b_a\ar@{.>}[d]^{\xi_a^{*}} & \\
  & c_a  & &
   }}
\end{equation*}
\noindent Moreover, $\langle c_a, \xi_{a}^{*} : b_a \rightarrow c_a\rangle$ is a $C$-reflection of $b_a$.
\end{lemma}

\begin{lemma}\label{key lemma 2} (\cite[Lemma 6.4]{SBWX-2021}
Suppose that the Sierpinski space $\Sigma~\!\!2$ is a $\mathbf{K}$-space and a $T_0$ space $X$ has a $\mathbf{K}$-reflection $\langle X^{k},\eta_{X} \rangle$. Then $\eta_{X}$ is a dense topological embedding.
\end{lemma}

\begin{theorem}(\cite[Lemma 6.7]{SBWX-2021})\label{key theorem 1}  Let $\mathbf{K}$ be a full subcategory of $\mathbf{Top}_0$ which is closed with respect to homeomorphisms and contains $\mathbf{Sob}$. For a $T_0$ space $X$, suppose that the following conditions are satisfied:
\begin{enumerate}[\rm (1)]
\item $X$ is not a $\mathbf{K}$-space.
\item $X\in \ir_c(X)$.
\item $\langle X_{\top}^{\flat},\eta_{X}^{\flat} \rangle$ is a sobrification of $X$.
\item The canonical sobrification $X^s$ (and hence any sobrification) of $X$ is not a $\mathbf{K}$-reflection of $X$.
\end{enumerate}
\noindent Then the $\mathbf{K}$-reflection of $X$ does not exist.
\end{theorem}

\begin{corollary} (\cite[Corollary 6.8]{SBWX-2021})\label{key corollary 1}
Let $\mathbf{K}$ be a full subcategory of $\mathbf{Top}_0$ which is closed with respect to homeomorphisms and contains $\mathbf{Sob}$. For a $T_0$ space $X$, suppose that the following conditions are satisfied:
\begin{enumerate}[\rm (a)]
\item $X$ is not a $\mathbf{K}$-space.
\item $\ir_c(X)=\{\overline{\{x\}} : x\in X\}\bigcup\{X\}$.
\item The canonical sobrification $X^s$ (and hence any sobrification) of $X$ is not a $\mathbf{K}$-reflection of $X$.
\end{enumerate}
\noindent Then by conditions (a) and (b), $\langle X_{\top}^{\flat},\eta_{X}^{\flat} \rangle$ is a sobrification of $X$, and hence the $\mathbf{K}$-reflection of $X$ does not exist.
\end{corollary}

By Remark \ref{four categories closed with respect to homeomorphisms}, Theorem \ref{K-adequate reflective}, Theorem \ref{four categories adequate},  Theorem \ref{key theorem 1} and Corollary \ref{key corollary 1}, we get the following corollaries.

\begin{corollary}\label{sobrification agree K-reflection 1}  Let $\mathbf{K}$ be a full and adequate subcategory of $\mathbf{Top}_0$ which is closed with respect to homeomorphisms and contains $\mathbf{Sob}$. For a $T_0$ space$X$, suppose that the following conditions are satisfied:
\begin{enumerate}[\rm (1)]
\item $X$ is not a $\mathbf{K}$-space.
\item $X\in \ir_c(X)$.
\item $\langle X_{\top}^{\flat},\eta_{X}^{\flat} \rangle$ is a sobrification of $X$.
\end{enumerate}
\noindent Then the sobrification of $X$ and the $\mathbf{K}$-reflection of $X$ coincide.
\end{corollary}

\begin{corollary}\label{sobrification agree K-reflection 2}  Suppose that $\mathbf{K}=\mathbf{Top}_d$ or $\mathbf{K}=\mathbf{Top}_w$ or $\mathbf{K}$ is a Keimel-Lawson category. For a $T_0$ space $X$, assume that the following conditions are satisfied:
\begin{enumerate}[\rm (1)]
\item $X$ is not a $\mathbf{K}$-space.
\item $X\in \ir_c(X)$.
\item $\langle X_{\top}^{\flat},\eta_{X}^{\flat} \rangle$ is a sobrification of $X$.
\end{enumerate}
\noindent Then the sobrification of $X$ and the $\mathbf{K}$-reflection of $X$ coincide.
\end{corollary}

\begin{corollary} \label{sobrification agree K-reflection 3}
Let $\mathbf{K}$ be a full and adequate subcategory of $\mathbf{Top}_0$ which is closed with respect to homeomorphisms and contains $\mathbf{Sob}$. For a $T_0$ space$X$, suppose that the following two conditions are satisfied:
\begin{enumerate}[\rm (a)]
\item $X$ is not a $\mathbf{K}$-space.
\item $\ir_c(X)=\{\overline{\{x\}} : x\in X\}\bigcup\{X\}$.
\end{enumerate}
\noindent Then $\langle X_{\top}^{\flat},\eta_{X}^{\flat} \rangle$ is a sobrification of $X$, and the sobrification of $X$ and the $\mathbf{K}$-reflection of $X$ coincide.
\end{corollary}

\begin{corollary} \label{sobrification agree K-reflection 4}
Suppose that $\mathbf{K}=\mathbf{Top}_d$ or $\mathbf{K}=\mathbf{Top}_w$ or $\mathbf{K}$ is a Keimel-Lawson category. For a $T_0$ space $X$, assume that the following two conditions are satisfied:
\begin{enumerate}[\rm (a)]
\item $X$ is not a $\mathbf{K}$-space.
\item $\ir_c(X)=\{\overline{\{x\}} : x\in X\}\bigcup\{X\}$.
\end{enumerate}
\noindent Then $\langle X_{\top}^{\flat},\eta_{X}^{\flat} \rangle$ is a sobrification of $X$, and the sobrification of $X$ and the $\mathbf{K}$-reflection of $X$ coincide.
\end{corollary}

Similarly, we have the following conclusions.

\begin{theorem}\label{key theorem 2}  Let $\mathbf{K}$ be a full subcategory of $\mathbf{Top}_0$ which is closed with respect to homeomorphisms and contains $\mathbf{Sob}$. For a $\mathbf{K}^{\neg}$-space $X$, suppose that the following two conditions are satisfied:
\begin{enumerate}[\rm (1)]
\item $\langle X_{\top}^{\natural},\eta_{X}^{\natural} \rangle$ is a sobrification of $X$.
\item The canonical sobrification $X^s$ (and hence any sobrification) of $X$ is not a $\mathbf{K}$-reflection of $X$.
\end{enumerate}
\noindent Then the $\mathbf{K}$-reflection of $X$ does not exist.
\end{theorem}

\begin{proof} First, we show that $(\mathcal C(X)\setminus \{X\setminus \{\top_X\}\})\bigcup \{X\cup\{\top\}\}$ (as the set of all closed sets) is a $T_0$ topology on $X\cup\{\top\}$. Suppose that $x, y\in X\cup\{\top\}$ with $x\neq y$.

 {Case 1.} $x\in X$ and $y=\top$.

 Then $\cl_{X_{\top}^{\natural}}\{y\}=X\cup \{\top\}\neq \cl_{X_{\top}^{\natural}}\{x\}$ since $\cl_{X_{\top}^{\natural}}\{x\}\subseteq X\in \mathcal C(X_{\top}^{\natural})$.

 {Case 2.} $x\in X\setminus \{\top_X\}$ and $y=\top_X$.

 As $X$ is a $\mathbf{K}^{\neg}$-space, $\cl_{X_{\top}^{\natural}}\{y\}=X$ and $\cl_{X}\{x\}\neq X\setminus \{\top_X\}$, whence
$\cl_X\{x\}\in \mathcal C(X)\setminus \{X\setminus \{\top_X\}\}$. Therefore, $\cl_{X_{\top}^{\natural}}\{x\}=\cl_X\{x\}\neq X= \cl_{X_{\top}^{\natural}}\{y\}$.

 {Case 3.} $x, y\in X\setminus \{\top_X\}$.

 Since $X$ is a $\mathbf{K}^{\neg}$-space, $\cl_{X_{\top}^{\natural}}\{x\}=\cl_X\{x\}\neq \cl_X\{y\}=\cl_{X_{\top}^{\natural}}\{y\}$.

 Thus $X_{\top}^{\natural}$ is a $T_0$ space.

Now suppose, on the contrary, that the $\mathbf{K}$-reflection $\langle X^{k},\eta_{X}\rangle$ of $X$ exist. Then by condition (1) (note that every sober space is a $\mathbf{K}$-space and hence $\Sigma~\!\!2$ is a $\mathbf{K}$-space since $\mathbf{Sob}\subseteq_{full} \mathbf{K}$), there is a unique continuous mapping $(\eta_{X}^{\natural})^{k} : X^k \rightarrow X_\top^{\natural}$ such that $(\eta_{X}^{\natural})^{k}\circ \eta_{X}=\eta_{X}^{\natural}$, that is, the the following diagram commutes.

\begin{equation*}
\centerline{
\xymatrix{ X \ar[dr]_{\eta_X^\natural} \ar[r]^-{\eta_{X}}&  {X^{k}} \ar@{.>}[d]^{(\eta_{X}^{\natural})^{k}} & \\
  & X_{\top}^{\natural}  & &
   }}
\end{equation*}

\noindent Since $\mathbf{Sob}\subseteq_{full} \mathbf{K}\subseteq_{full} \mathbf{Top}_0$, it follows from Lemma \ref{key lemma 1} that $\langle X_{\top}^{\natural},(\eta_{X}^{\natural})^{k} \rangle$ is a sobrification of $X^{k}$. By Lemma \ref{key lemma 2}, $\eta_{X}$, $\eta_{X}^{\natural}$ and $(\eta_{X}^{\natural})^{k}$ are all dense topological embeddings. As $(\eta_{X}^{\natural})^{k}(\eta_{X}(X))=\eta_{X}^{\natural}(X)=(X\setminus \{\top_X\})\cup \{\top\}$, we have that $(\eta_{X}^{\natural})^{k}(X^{k})=(X\setminus \{\top_X\})\cup \{\top\}$ or $(\eta_{X}^{\natural})^{k}(X^{k})=X\cup\{\top\}$.

{Case A.} $(\eta_{X}^{\natural})^{k}(X^{k})=(X\setminus \{\top_X\})\cup \{\top\}$.

Then $X^{k}$ is homeomorphic to the subspace $(X\setminus \{\top_X\})\cup \{\top\}$ of $X_{\top}^{\natural}$ since $(\eta_{X}^{\natural})^{k}$ is a topological embedding. Clearly, $X$ is homeomorphic to the subspace $(X\setminus \{\top_X\})\cup \{\top\}$ via the homeomorphism $\varphi : X \rightarrow (X\setminus \{\top_X\})\cup \{\top\}$ defined by $\varphi(x)=x$ for each $x\in X\setminus \{\top_X\}$ and $\varphi(\top_X)=\top$. Whence $X$ is homeomorphic to $X^k$. As $\mathbf{K}$ is closed with respect to homeomorphisms, $X$ is a $\mathbf{K}$-space, which contradicts the hypothesis that $X$ is a $\mathbf{K}^{\neg}$-space.

{Case B.} $(\eta_{X}^{\natural})^{k}(X^{k})=X\cup\{\top\}$.

Then $(\eta_{X}^{\natural})^{k} : X^{k} \rightarrow X_{\top}^\natural$ is a homeomorphism. So $\langle X_{\top}^{\natural},\eta_{X}^{\natural}=(\eta_{X}^{\natural})^{k}\circ \eta_{X} \rangle$, as a sobrification of $X$ by condition (1), is also a $\mathbf{K}$-reflection of $X$, contradicting condition (2).

Therefore, the $\mathbf{K}$-reflection of $X$ does not exist.
\end{proof}

\begin{corollary}\label{key corollary 2}
Let $\mathbf{K}$ be a full subcategory of $\mathbf{Top}_0$ which is closed with respect to homeomorphisms and contains $\mathbf{Sob}$. For a $\mathbf{K}^{\neg}$-space $X$, suppose that the following two conditions are satisfied:
\begin{enumerate}[\rm (a)]
\item $\ir_c(X)=\{\overline{\{x\}} : x\in X\}\bigcup\{X\setminus \{\top_X\}\}$.
\item The sobrification $X^s$ (and hence any sobrification) of $X$ is not a $\mathbf{K}$-reflection of $X$.
\end{enumerate}
\noindent Then by condition (a), $\langle X_{\top}^{\natural},\eta_{X}^{\natural} \rangle$ is a sobrification of $X$, and hence the $\mathbf{K}$-reflection of $X$ does not exist.
\end{corollary}
\begin{proof} Since $X$ is not a $\mathbf{K}$-space, $X$ is not sober, whence by condition (a), $X\setminus \{\top_X\}\neq\overline{\{x\}}$ for each $x\in X$. It is well-known that the space $X^s=P_H(\ir_c(X))$ with the canonical mapping $\eta_{X}: X\longrightarrow X^s$, $\eta_{X}(x)=x$, is a sobrification of $X$ (see, for example, \cite[Exercise V-4.9]{redbook}). For $C\in \mathcal C(X)$, by condition (a) we have

$$\Box_{\ir_c(X)} C=\{A\in \ir_c(X) : A\subseteq  C\}=
\begin{cases}
\{\overline{\{c\}} : c\in C\},& C\neq X\setminus \{\top_X\}, C\neq X,\\
\{\overline{\{c\}} : c\in X\setminus \{\top_X\}\}\bigcup \{X\setminus \{\top_X\}\},& C=X\setminus \{\top_X\},\\
\{\overline{\{x\}} : x\in X\}\bigcup\{X\setminus \{\top_X\}\},& C=X.
	\end{cases}$$

\noindent Define a mapping $h : X^s\rightarrow X_{\top}^{\natural}$ by
$$h(A)=
\begin{cases}
	x & A=\overline{\{x\}}, x\in X\setminus \{\top_X\},\\
\top_X & A=X\setminus \{\top_X\},\\
\top & A=\overline{\{\top_X\}}=X.
	\end{cases}$$
\noindent Clearly, $h$ is a bijection. For each $B\in \mathcal C(X)$ and $C\in (\mathcal C(X)\setminus \{X\setminus \{\top_X\}\})\bigcup \{X\cup\{\top\}\}$, we have

$$h(\Box_{\ir_c(X)} B)=
\begin{cases}
	B& B\neq X\setminus \{\top_X\},B\neq X,\\
X& B=X\setminus\{\top_X\},\\
X\cup\{\top\}& B=X.
	\end{cases}$$
and
$$h^{-1}(C)=
\begin{cases}
	\Box_{\ir_c(X)}C& C\neq X, C\neq X\cup\{\top\},\\
\Box_{\ir_c(X)} (X\setminus \{\top_X\})& C=X,\\
\Box_{\ir_c(X)} X=\ir_c(X)& C=X\cup\{\top\}.
	\end{cases}$$

\noindent Thus $h$ is a homeomorphism, and hence $\langle X_{\top}^{\natural},\eta_{X}^{\natural}=h\circ \eta_X \rangle$ is a sobrification of $X$.

Therefore, conditions (1) and (2) of Theorem \ref{key theorem 2} hold for $X$. By Theorem \ref{key theorem 2}, the $\mathbf{K}$-reflection of $X$ does not exist.
\end{proof}

By Remark \ref{four categories closed with respect to homeomorphisms}, Theorem \ref{K-adequate reflective}, Theorem \ref{four categories adequate},  Theorem \ref{key theorem 2} and Corollary \ref{key corollary 2}, we have the following corollaries.

\begin{corollary}\label{sobrification agree K-reflection 11}  Let $\mathbf{K}$ be a full and adequate subcategory of $\mathbf{Top}_0$ which is closed with respect to homeomorphisms and contains $\mathbf{Sob}$. For a $\mathbf{K}^{\neg}$-space $X$, if $\langle X_{\top}^{\natural},\eta_{X}^{\natural} \rangle$ is a sobrification of $X$, then the sobrification of $X$ and the $\mathbf{K}$-reflection of $X$ coincide.
\end{corollary}

\begin{corollary}\label{sobrification agree K-reflection 12}  Suppose that $\mathbf{K}=\mathbf{Top}_d$ or $\mathbf{K}=\mathbf{Top}_w$ or $\mathbf{K}$ is a Keimel-Lawson category. For a $\mathbf{K}^{\neg}$-space $X$, if $\langle X_{\top}^{\natural},\eta_{X}^{\natural} \rangle$ is a sobrification of $X$, then the sobrification of $X$ and the $\mathbf{K}$-reflection of $X$ coincide.
\end{corollary}

\begin{corollary} \label{sobrification agree K-reflection 13}
Let $\mathbf{K}$ be a full and adequate subcategory of $\mathbf{Top}_0$ which is closed with respect to homeomorphisms and contains $\mathbf{Sob}$. For a $\mathbf{K}^{\neg}$-space $X$, if $\ir_c(X)=\{\overline{\{x\}} : x\in X\}\bigcup\{X\setminus \{\top_X\}\}$, then $\langle X_{\top}^{\natural},\eta_{X}^{\natural} \rangle$ is a sobrification of $X$, and the sobrification of $X$ and the $\mathbf{K}$-reflection of $X$ coincide.
\end{corollary}

\begin{corollary} \label{sobrification agree K-reflection 14}
Suppose that $\mathbf{K}=\mathbf{Top}_d$ or $\mathbf{K}=\mathbf{Top}_w$ or $\mathbf{K}$ is a Keimel-Lawson category. For a $\mathbf{K}^{\neg}$-space $X$, if $\ir_c(X)=\{\overline{\{x\}} : x\in X\}\bigcup\{X\setminus \{\top_X\}\}$, then $\langle X_{\top}^{\natural},\eta_{X}^{\natural} \rangle$ is a sobrification of $X$, and the sobrification of $X$ and the $\mathbf{K}$-reflection of $X$ coincide.
\end{corollary}

The following example shows that there is a $T_0$ space $X$ such that the $\mathbf{Cut}$-reflection of $X$ does not exist, and neither the $\mathbf{QSob}$-reflection of $X$ nor the $\mathbf{QSob}$-reflection of $X$ exists.

\begin{example}\label{example:3.10} Suppose that $\mathbf{K}\in \{\mathbf{QSob}, \mathbf{WSob}, \mathbf{Top}_{cut}\}$. Let $L=\mathbb{N}\cup \{\top_{\mathbb{N}}\}$. Define an order on $L$ by $1<2<3<...<n<n+1<...\top_{\mathbb{N}}$ and endow $L$ with the topology (as the set of all closed sets) $\tau=\{\da n : n\in \mathbb{N}\}\cup\{\emptyset, L\}\cup \{\mathbb{N}\}$ (clearly, the set of all open subsets is $\upsilon (L)\cup \{\{\top_{\mathbb{N}}\}\}=\sigma (L)\cup \{\{\top_{\mathbb{N}}\}\}$). Then
\begin{enumerate}[\rm (a)]

\item $(L,\tau)$ is not a cut space. So $(L,\tau)$ is neither a weakly sober nor a quasisober space.

Since $\mathbb{N}$ is closed and $\mathbb{N}^{\delta}=(\mathbb{N}^{\ua})^{\da}=\{\top_{\mathbb{N}}\}^{\da}=L\neq \mathbb{N}$, $(L,\tau)$ is not a cut space and hence by Lemma \ref{lemma:2.5},  $(L,\tau)$ is neither a weakly sober nor a quasisober space.

\item $(L,\tau)$ has a greatest element $\top_{(L,\tau)}$.

Clearly, the specialization order of $(L,\tau)$ agrees with the original order on $L$. Whence $(L,\tau)$ has a greatest element $\top_{(L,\tau)}$, namely, the element $\top_{\mathbb{N}}$.

 \item $\ir_c((L,\tau))=\{\downarrow x : x\in L\}\cup\{\mathbb{N}\}=\{\downarrow x : x\in L\}\cup\{L\setminus\{\top_{(L,\tau)}\}$ (note that $L=\da \top_{\mathbb{N}}$). Whence $L\setminus \{\top_{\mathbb{N}}\}=\mathbb{N}\in\ir_c((L,\tau))$.

\item $\overline{\{x\}}=\da x\neq \mathbb{N}=L\setminus \{\top_{\mathbb{N}}\}$ for each $x\in L$.

\item $(L,\tau)$ is a $\mathbf{K}^{\neg}$-space and $\ir_c((L,\tau))=\{\downarrow x : x\in L\}\cup\{\mathbb{N}\}=\{\downarrow x : x\in L\}\cup\{L\setminus\{\top_{(L,\tau)}\}$.

By (a)-(d), $(L,\tau)$ is a $\mathbf{K}^{\neg}$-space and $\ir_c((L,\tau))=\{\downarrow x : x\in L\}\cup\{\mathbb{N}\}=\{\downarrow x : x\in L\}\cup\{L\setminus\{\top_{(L,\tau)}\}$.

\item $\langle (L,\tau)_{\top}^{\natural},\eta_{L}^{\natural} \rangle$ is a sobrification of $(L,\tau)$, where $\eta_{L}^{\natural}: (L,\tau)\rightarrow  (L,\tau)_{\top}^{\natural}$ is defined by  $\eta_{L}^{\natural}(x)=x$ for each $x\in \mathbb{N}$ and $\eta_{L}^{\natural}(\top_{\mathbb{N}})=\top$.

By (e) and Corollary \ref{key corollary 2}, $\langle (L,\tau)_{\top}^{\natural},\eta_{L}^{\natural} \rangle$ is a sobrification of $(L,\tau)$.

\item $\langle (L,\tau)_{\top}^{\natural},\eta_{L}^{\natural} \rangle$ is not a $\mathbf{K}$-reflection of $(L,\tau)$.

Assume, on the contrary, that $\langle (L,\tau)_{\top}^{\natural},\eta_{L}^{\natural} \rangle$ is a $\mathbf{K}$-reflection of $(L,\tau)$. Let $\mathbb{N}_{\top_1\top_2}=\mathbb{N}\cup \{\top_1, \top_2\}$. Define an order on $\mathbb{N}_{\top_1\top_2}$ by $n<n+1$ and $n<\top_1, n< \top_2$ for any $n\in \mathbb{N}$. Endow $\mathbb{N}_{\top_1\top_2}$ with the upper topology $\upsilon (\mathbb{N}_{\top_1\top_2})$. It is straightforward to verify that the set of all nonempty closet subsets of $(\mathbb{N}_{\top_1\top_2}, \upsilon (\mathbb{N}_{\top_1\top_2}))$ is $\{\da n : n\in \mathbb{N}\}\cup\{\da \top_1, \da \top_2\}\cup\{\mathbb{N}\}$, whence $\ir_c (\mathbb{N}_{\top_1\top_2}, \upsilon (\mathbb{N}_{\top_1\top_2}))=\{\da x : x\in \mathbb{N}_{\top_1\top_2}\}\cup \{\mathbb{N}\}$. Since $\mathbb{N}^{\delta}=(\mathbb{N}^{\ua})^{\da}=\{\top_1, \top_2\}^{\da}=\mathbb{N}$, $(\mathbb{N}_{\top_1\top_2}, \upsilon (\mathbb{N}_{\top_1\top_2}))$ is a quasisober space, whence by Lemma \ref{lemma:2.5}, it is a $\mathbf{K}$-space.

Define a mapping $f:(L,\tau)\rightarrow (\mathbb{N}_{\top_1\top_2}, \upsilon (\mathbb{N}_{\top_1\top_2}))$ by

$$f(x)=
\begin{cases}
	n & x=n\in \mathbb{N},\\
    \top_1& x=\top_{\mathbb{N}}.
	\end{cases}$$

It is easy to see that $f$ is a topological embedding. Since $\langle (L,\tau)_{\top}^{\natural},\eta_{L}^{\natural} \rangle$ is a $\mathbf{K}$-reflection of $(L,\tau)$, there is a unique $f^k: (L,\tau)_{\top}^{\natural}\rightarrow ((\mathbb{N}_{\top_1\top_2}, \upsilon (\mathbb{N}_{\top_1\top_2}))$ such that $f=f^k\circ \eta_{L}^{\natural}$, that is, the the following diagram commutes.

\begin{equation*}
\centerline{
\xymatrix{ (L,\tau) \ar[dr]_{f} \ar[r]^-{\eta_{L}^{\natural}}&  (L,\tau)_{\top}^{\natural}\ar@{.>}[d]^{f^{k}} & \\
  & (\mathbb{N}_{\top_1\top_2}, \upsilon (\mathbb{N}_{\top_1\top_2}))  & &
   }}
\end{equation*}

\noindent Then $f^k(n)=f^k(\eta_{L}^{\natural}(n))=f(n)=n$ for each $n\in \mathbb{N}$ and $f^k(\top)=f^k(\eta_{L}^{\natural}(\top_{\mathbb{N}}))=f(\top_{\mathbb{N}})=\top_1$. For each $n\in \mathbb{N}$, since $n<\top_{\mathbb{N}}<\top$ in $(L,\tau)_{\top}^{\natural}$, we have that $n=f^k(n)\leq f^k(\top_{\mathbb{N}})\leq f^k(\top)=\top_1$ in $(\mathbb{N}_{\top_1\top_2}, \upsilon (\mathbb{N}_{\top_1\top_2}))$. Whence $f^k(\top_{\mathbb{N}})$ is an upper bound of $\mathbb{N}$ in $\mathbb{N}_{\top_1\top_2}$ and $f^k(\top_{\mathbb{N}})\leq \top_1$. So $f^k(\top_{\mathbb{N}})=\top_1$ and hence $(f^k)^{-1}(\mathbb{N})=\mathbb{N} \notin \mathcal{C}((L,\tau)_{\top}^{\natural})$ (note that $\mathbb{N}$ is closed in $(\mathbb{N}_{\top_1\top_2}, \upsilon (\mathbb{N}_{\top_1\top_2}))$), which contradicts the continuity of $f^k$.

Thus $\langle (L,\tau)_{\top}^{\natural},\eta_{L}^{\natural} \rangle$ is not a $\mathbf{K}$-reflection of $(L,\tau)$.

\item The $\mathbf{K}$-reflection of $(L,\tau)$ does not exist.

By (e)(f)(g) and Corollary \ref{key corollary 2}, the $\mathbf{K}$-reflection of $(L,\tau)$ does not exist.

\item $(L,\tau)$ is a $\mathsf{DC}$-space and hence it is both a Rudin space and a $\mathsf{WD}$-space.

Since $\ir_c((L,\tau))=\{\downarrow x : x\in L\}\cup\{\mathbb{N}\}$ and $\mathbb{N}$ is a chain, $(L,\tau)$ is a $\mathsf{DC}$-space, whence it is both a Rudin space and a $\mathsf{WD}$-space by Corollary \ref{SDRWspacerelation}.

\item $(L,\tau)$ is neither  a $d$-space nor a well-filtered space.

Since $\{\top_{\mathbb{N}}\}$ is open in $(L,\tau)$ but $\{\top_{\mathbb{N}}\}\not\in\sigma (L)$, $(L,\tau)$ is not a $d$-space and hence not a well-filtered space.

\item The sobrification of $(L,\tau)$, the $\mathbf{WF}$-reflection of $(L, \tau)$ and the $d$-reflection of $(L, \tau)$ all coincide.

By (b)(c)(d)(j) and Corollary \ref{sobrification agree K-reflection 14}, the sobrification of $(L,\tau)$, the $\mathbf{WF}$-reflection of $(L, \tau)$ and the $d$-reflection of $(L, \tau)$ all coincide. It follows from (f) that $\langle (L,\tau)_{\top}^{\natural},\eta_{L}^{\natural} \rangle$ is both a $\mathbf{WF}$-reflection of $(L, \tau)$ and a $d$-reflection of $(L, \tau)$.

\end{enumerate}
\end{example}

The following example shows that there is a $T_0$ space $Y$ such that the $\mathbf{DC}$-reflection of $Y$ nor the $\mathbf{QSob}$-reflection of $Y$ exists.

\begin{example} \label{example:3.7} Suppose that $\mathbf{K}\in \{\mathbf{Top}_{dc}, \mathbf{QSob}\}$.
Let $X$ be a countable infinite set and endow $X$ with the cofinite topology (having the complements of the finite sets as open sets). The
resulting space is denoted by $X_{cof}$. Then
\begin{enumerate}[\rm (a)]

     \item $\mathcal C(X_{cof})=\{\emptyset, X\}\cup X^{(<\omega)}$, $X_{cof}$ is $T_1$ and hence a $d$-space.
    \item $\mk (X_{cof})=2^X\setminus \{\emptyset\}$, $X_{cof}$ is locally compact and first-countable, and $X_{cof}$ is not well-filtered and hence not sober.

Clearly, $\mk (X_{cof})=2^X\setminus \{\emptyset\}$ and hence $X_{cof}$ is locally compact. For $x\in X$, the countable family $\{(X\setminus F)\cup \{x\} : F\in X^{(<\omega)}\}$ is an open neighborhood base of $x$ in $X_{cof}$. Whence $X_{cof}$ is first-countable. Let $\mathcal K=\{X\setminus F: F\in X^{(<\omega)}\}$. Then $\mathcal K$ is a filtered family of saturated compact subsets of $X_{cof}$ and $\bigcap \mathcal K=\emptyset$, but $X\setminus F\neq\emptyset$ for every $ F\in X^{(<\omega)}$. Thus $X_{cof}$ is not well-filtered.

 \item $\ir_c(X_{cof})=\{\{x\}:x\in X\}\cup \{X\}$.

    Since $X$ is infinite, $X$ is irreducible. Whence $\ir_c(X_{cof})=\{\{x\}:x\in X\}\cup \{X\}$.

    \item $X_{cof}$ is a Rudin space and hence a $\mathsf{WD}$-space.

By (c) and Lemma \ref{DRWIsetrelation} we only need to show that $X$ is a Rudin set. Let $\mathcal{K}=\{X\setminus F : F\in X^{(<\omega)}\}$. Then by (b), $\mathcal{K}\subseteq \mk (X_{cof})$ is a filtered family and $X\in M(\mathcal K)$. If $B\in \mathcal C(X_{cof})$ is a proper subset of $X$, then $B$ is finite, whence $B\cap (X\setminus B=\emptyset$ and $X\setminus B\in \mathcal K$, proving that $B\not\in M(\mathcal K)$. So $X\in m(\mathcal K)$ and hence $X$ is a Rudin set. Thus $X_{cof}$ is a Rudin space and hence a $\mathsf{WD}$-space by Corollary \ref{SDRWspacerelation}.

    \item $X_{cof}$ is not a $\mathbf{K}$-space.

    By $\ir_c(X_{cof})=\{\{x\}:x\in X\}\cup \{X\}$ and $X^{\delta}=X$, $X_{cof}$ is a weakly sober space and hence a cut space. Clearly, $\mathcal D(X_{cof})=\{\{x\}:x\in X\}$, whence $X\neq D^{\delta}$ and $X\neq \cl_{X_{cof}}D$ for each $D\in \mathcal D(X_{cof})$. So $X_{cof}$ is neither a $\mathsf{DC}$ space nor a quasisober space.

\item $\langle (X_{cof})_{\top}^{\flat},\eta_{X_{cof}}^{\flat} \rangle$ is a sobrification of $X_{cof}$.

     By (c)(e) and Corollary \ref{key corollary 1}, $\langle (X_{cof})_{\top}^{\flat},\eta_{X_{cof}}^{\flat} \rangle$ is a sobrification of $X_{cof}$.

     \item  $\langle (X_{cof})_{\top}^{\flat},\eta_{X_{cof}}^{\flat} \rangle$ is not a $\mathbf{K}$-reflection of $X_{cof}$.

     We prove that by contradiction. Assume that $\langle (X_{cof})_{\top}^{\flat},\eta_{X_{cof}}^{\flat} \rangle$ is a $\mathbf{K}$-reflection of $X_{cof}$. Let $\mathbb{N}$ be the poset of all natural numbers with the usual order (i.e., $n< n+1$ for each $n\in \mathbb{N}$) and let $Y=X\cup \mathbb{N}$ with ordering defined by $x\leq_Y y$ if{}f $x=y$ or $x, y\in \mathbb{N}$ and $x \leq_{\mathbb{N}} y$. Endow $Y$ with the upper topology $\upsilon(Y)$. It is straightforward to verify that $\ir_c((Y,\upsilon(Y)))=\{\downarrow x:x\in Y\}\cup \{Y\}$ and $\mathbb{N}^{\delta}=Y=\cl_{\upsilon(Y)}\mathbb{N}$. Whence $(Y, \upsilon(Y))$ is both a $\mathsf{DC}$ space and a quasisober space, and hence a $\mathbf{K}$-space.

     Define a mapping $f: X_{cof}\rightarrow (Y, \upsilon (Y))$ by $f(x)=x$ for each $x\in X$. It is easy to verify that $f$ is continuous (in fact, $f$ is a topological embedding). Since $\langle (X_{cof})_{\top}^{\flat},\eta_{X_{cof}}^{\flat} \rangle$ is a $\mathbf{K}$-reflection of $X_{cof}$, there exists a unique continuous mapping $f^{k}: (X_{cof})_{\top}^{\flat}\rightarrow (Y, \upsilon (Y))$ such that $f^{k}\circ \eta_{X_{cof}}^{\flat}=f$, that is, the following diagram commutes.

\begin{equation*}
\centerline{
\xymatrix{ X_{cof} \ar[dr]_{f} \ar[r]^-{\eta_{X_{cof}}^{\flat}}&  (X_{cof})_{\top}^{\flat}\ar@{.>}[d]^{f^{k}} & \\
  & (Y, \upsilon (Y))  & &
   }}
\end{equation*}

For each $x\in X$, we have that $f^{k}(x)=f^{k}(\eta_{X_{cof}}^{\flat}(x))=f(x)=x$ and $x<\top$ in $(X_{cof})_{\top}^{\flat}$, whence $x=f^{k}(x)\leq f^k(\top)$ in $(Y, \upsilon (Y))$. It follows that $X\subseteq \cl_{\upsilon (Y)} \{f^k(\top)\}=\da_Y f^k(\top)$, that is, $f^k(\top)$ is an upper bound of $X$ in $Y$, a contraction.

Thus $\langle (X_{cof})_{\top}^{\flat},\eta_{X_{cof}}^{\flat} \rangle$ is not a $\mathbf{K}$-reflection of $X_{cof}$.

\item The $\mathbf{K}$-reflection of $X_{cof}$ does not exists.

By (c)(e)(f)(g) and Corollary \ref{key corollary 1}, the $\mathbf{K}$-reflection of $X_{cof}$ does not exists, namely, neither the $\mathbf{DC}$-reflection of $X_{cof}$ nor the $\mathbf{QSob}$-reflection of $X_{cof}$ exists.

\item The sobrification of $X_{cof}$ agrees with the $\mathbf{WF}$-reflection of $X_{cof}$.

By (b)(c) and Corollary \ref{sobrification agree K-reflection 4}, the sobrification of $X_{cof}$ agrees with the $\mathbf{WF}$-reflection of $X_{cof}$. It follows from (f) that $\langle (X_{cof})_{\top}^{\natural},\eta_{X_{cof}}^{\natural} \rangle$ is both a sobrification of $X_{cof}$ and a $\mathbf{WF}$-reflection of $X_{cof}$.

\end{enumerate}
\end{example}

Now we show that neither the $\mathbf{DC}$-reflection of Johnstone space nor the $\mathbf{QSob}$-reflection of Johnstone space exists.

\begin{example}\rm\label{example:3.8}  Suppose that $\mathbf{K}\in \{\mathbf{Top}_{dc}, \mathbf{QSob}\}$. Let $\mathbb{J}=\mathbb{N}\times (\mathbb{N}\cup \{\infty\})$ with ordering defined by $(j, k)\leq (m, n)$ if{}f $j = m$ and $k \leq n$, or $n =\infty$ and $k\leq m$ (see Figure 1).

\begin{figure}[ht]
	\centering
	\includegraphics[height=4.5cm,width=4.5cm]{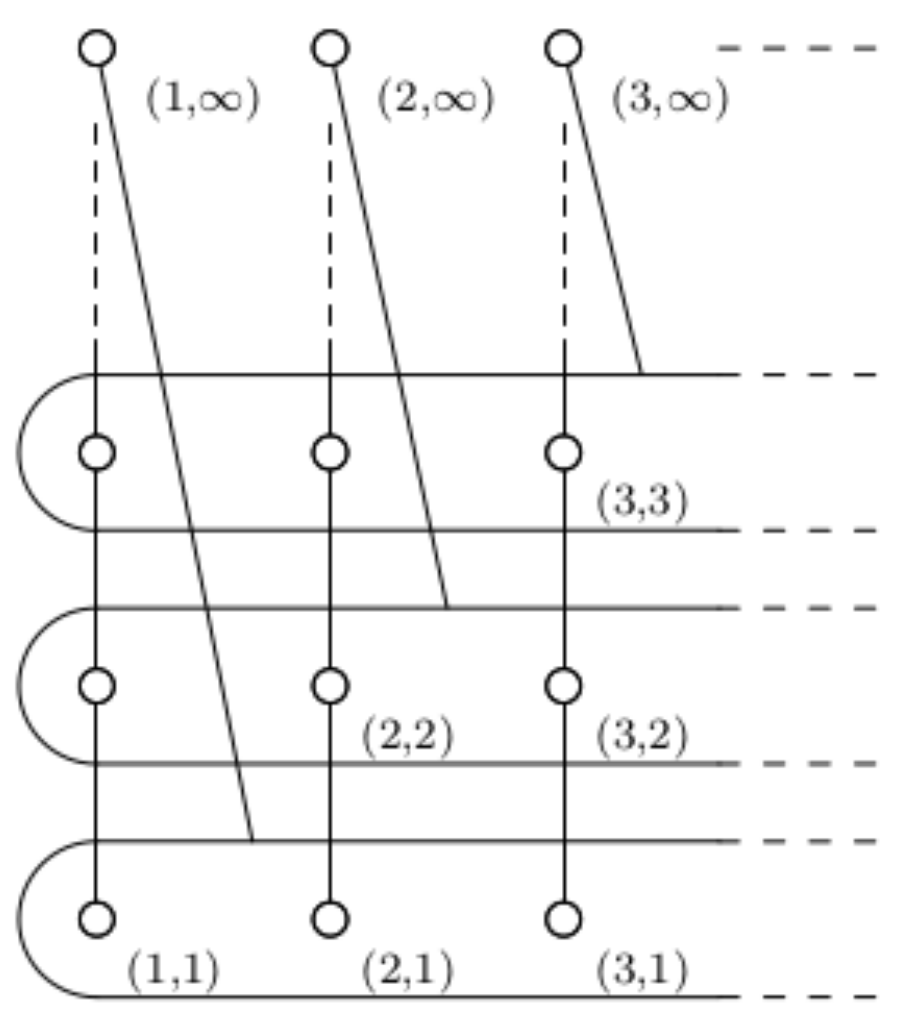}
	\caption{Johnstone's dcpo $\mathbb{J}$}
\end{figure}

\noindent The poset $\mathbb{J}$ is a well-known dcpo constructed by Johnstone in \cite{johnstone-81}. The set $\mathbb{J}_{max}=\{(n, \infty) : n\in\mn \}$ is the set of all maximal elements of $\mathbb{J}$. Since $\mathbb{J}$ is a dcpo, $\Sigma~\!\!\mathbb{J}$ is a $d$-space. The following three conclusions about $\Sigma~\!\!\mathbb{J}$ are known (see, for example, \cite[Example 3.1]{LL} and \cite[the proof of Lemma 3.1]{MLZ}):
\begin{enumerate}[\rm (i)]
\item $\ir_c (\Sigma~\!\!\mathbb{J})=\{\cl_{\sigma (\mathbb{J})}{\{x\}}=\da x : x\in \mathbb{J}\}\cup \{\mathbb{J}\}$.
\item $\mathsf{K}(\Sigma~\!\!\mathbb{J})=(2^{\mathbb{J}_{max}} \setminus \{\emptyset\})\bigcup \mathbf{Fin}~\!\mathbb{J}$.
\item $\Sigma~\!\!\mathbb{J}$ is not well-filtered and hence not a sober space.

Indeed, let $\mathcal{K}_{max}=\{\mathbb{J}_{max}\setminus F : F\in (\mathbb{J}_{max})^{(<\omega)}\}$. Then by (ii), $\mathcal{K}_{max}\subseteq \mathsf{K}(\Sigma~\!\!\mathbb{J})$ is a filtered family and $\bigcap\mathcal{K}_{max}=\bigcap_{F\in (\mathbb{J}_{max})^{(<\omega)}}  (\mathbb{J}_{max}\setminus F)=\mathbb{J}_{max}\setminus \bigcup (\mathbb{J}_{max})^{(<\omega)}=\emptyset$, but there is no $F\in (\mathbb{J}_{max})^{(<\omega)}$ with $\mathbb{J}_{max}\setminus F=\emptyset$. Therefore, $\Sigma~\!\!\mathbb{J}$ is not well-filtered.

\end{enumerate}
\noindent Moreover, we have
\begin{enumerate}[\rm (a)]
\item $\Sigma~\!\!\mathbb{J}$ is a Rudin space and hence a $\mathsf{WD}$ space.

Clearly, for every $x\in \mathbb{J}$, $\da x$ is a Rudin set ($\da x\in m(\{\ua x\})$). Now we show that $\mathbb{J}\in m(\mathcal{K}_{max})$. Obviously, $\mathbb{J}\in M(\mathcal{K}_{max})$. Suppose that $B$ is a nonempty proper closed subset of $\Sigma~\!\!\mathbb{J}$. Then there is $x\in \mathbb{J}\setminus B$.

{Case 1.} $x=(n, m)$ for some $(n, m)\in \mathbb{N}\times \mathbb{N}$.

 Then $B\cap \ua x=\emptyset$. Let $F=\{(1, \infty), (2, \infty), ..., (m-1, \infty)\}$. Then $\mathbb{J}_{max} \setminus F=\{(l, \infty) : l\in \mathbb{N} \hbox{~and~} m\leq l\}\subseteq \ua x$. Whence $B\cap (\mathbb{J}_{max} \setminus F)=\emptyset$, and consequently, $B\not\in M(\mathcal{K}_{max})$.

{Case 2.} $x=(n, \infty)$ for some $n\in \mathbb{N}$.

As $\{(n, l) : l\in \mathbb{N}\}$ is a chain and $\vee_{\mathbb{J}} \{(n, l) : l\in \mathbb{N}\}=(n, \infty)=x\in \mathbb{J}\setminus B\in \sigma (\mathbb{J})$, there is $m\in \mathbb{N}$ such that $(n, m)\in \mathbb{J}\setminus B$. Then as shown in Case 1, $B\not\in M(\mathcal{K}_{max})$.

Thus $\mathbb{J}\in m(\mathcal{K}_{max})$. So $\Sigma~\!\!\mathbb{J}$ is a Rudin space.

\item $\Sigma~\!\!\mathbb{J}$ is a weakly sober space and hence a cut space, but $\Sigma~\!\!\mathbb{J}$ is not a $\mathbf{K}$-space.

Since $\ir_c (\Sigma~\!\!\mathbb{J})=\{\da x : x\in \mathbb{J}\}\cup \{\mathbb{J}\}$ and $\mathbb{J}^{\delta}=\mathbb{J}$, $\Sigma~\!\!\mathbb{J}$ is a weakly sober space and hence a cut space. For any $D\in \mathcal D(\mathbb{J})$, we have $\cl_{\sigma (\mathbb{J})}D=\da \vee D$ and $D^{\delta}=\da \vee D$, whence $\mathbb{J}\neq \cl_{\sigma (\mathbb{J})}D$ and $\mathbb{J}\neq D^{\delta}$, showing that $\Sigma~\!\!\mathbb{J}$ is neither a $\mathsf{DC}$ space nor a quasisober space, that is, $\Sigma~\!\!\mathbb{J}$ is not a $\mathbf{K}$-space.

\item The canonical sobrification $\langle (\Sigma~\!\!\mathbb{J})^s=P_H(\ir_c(\Sigma~\!\!\mathbb{J})), \eta_{\mathbb{J}}\rangle$ of $\Sigma~\!\!\mathbb{J}$ is not a $\mathbf{K}$-reflection of $\Sigma~\!\!\mathbb{J}$, where $\eta_{\mathbb{J}} : \Sigma~\!\!\mathbb{J}\rightarrow P_H(\ir_c(\Sigma~\!\!\mathbb{J}))$ is defined by $\eta_{\mathbb{J}}(x)=\cl_{\sigma (\mathbb{J})}\{x\}=\da x$ for each $x\in \mathbb{J}$.

Assume, on the contrary, that $\langle P_H(\ir_c(\Sigma~\!\!\mathbb{J})), \eta_{\mathbb{J}}\rangle$ is a a $\mathbf{K}$-reflection of $\Sigma~\!\!\mathbb{J}$. Let $\mathbb{N}$ be the poset of all natural numbers with the usual order and consider the Scott space $\Sigma~\!\!\mathbb{N}$. Since $\sigma (\mathbb{N})=\{\ua n : n\in \mathbb{N}\}\cup \{\emptyset\}$, we have $\ir_c(\Sigma~\!\!\mathbb{N})=\{\da n : n\in \mathbb{N}\}\cup \{\mathbb{N}\}$ and $\mathbb{N}=\mathbb{N}^{\delta}$, and consequently, $\Sigma~\!\!\mathbb{N}$ is both a $\mathsf{DC}$ space and a quasisober space. Whence $\Sigma~\!\!\mathbb{N}$ is a $\mathbf{K}$-space.

Define a mapping $f : \Sigma~\!\!\mathbb{J}\rightarrow \Sigma~\!\! \mathbb{N}$ by
$$f((x, y))=
\begin{cases}
	min \{n,m\} & (x, y)=(n, m)\in \mathbb{N}\times \mathbb{N},\\
    n& (x, y)=(n, \infty).
	\end{cases}$$

 As $f^{-1}(\uparrow n)=\{(i,j): (i, j)\in \mathbb{N}\times \mathbb{N} \hbox{~and~} n\leq i, n\leq j\}\cup \{(i,\infty) : i\in \mathbb{N}\hbox{~and~} n\leq i\}=\mathbb{J}\setminus \bigcup\limits_{m=1}^{n-1}\da (m, \infty)\in \sigma(\mathbb{J})$ for each $n\in \mathbb{N}$, $f$ is continuous. Since $\langle P_H(\ir_c(\Sigma~\!\!\mathbb{J})), \eta_{\mathbb{J}}\rangle$ is a $\mathbf{K}$-reflection of $\Sigma~\!\!\mathbb{J}$, there is a unique continuous mapping $f^{k}:P_H(\ir_c(\Sigma~\!\!\mathbb{J}))\rightarrow \Sigma~\!\!\mathbb{N}$ such that $f^{k}\circ \eta_{\mathbb{J}}=f$, that is, the following diagram commutes.

\begin{equation*}
\centerline{
\xymatrix{ \Sigma~\!\!\mathbb{J} \ar[dr]_{f} \ar[r]^-{\eta_{\mathbb{J}}}&  P_H(\ir_c(\Sigma~\!\!\mathbb{J}))\ar@{.>}[d]^{f^{k}} & \\
  & \Sigma~\!\!\mathbb{N}  & &
   }}
\end{equation*}

Since $\downarrow x\subseteq \mathbb{J}$ for any $x\in \mathbb{J}$, we have $f(x)=f^k(\eta_{\mathbb{J}}(x))=f^k(\downarrow x)\leq f^k(\mathbb{J})$. Thus $n=f((n, n))\leq f^k(\mathbb{J})$ for any $n\in \mathbb{N}$, and hence $f^k(\mathbb{J})$ is a greatest element of $\mathbb{N}$, a contradiction.

Thus the canonical sobrification $\langle (\Sigma~\!\!\mathbb{J})^s=P_H(\ir_c(\Sigma~\!\!\mathbb{J})), \eta_{\mathbb{J}}\rangle$ of $\Sigma~\!\!\mathbb{J}$ is not a $\mathbf{DC}$-reflection of $\Sigma~\!\!\mathbb{J}$.

\item The $\mathbf{K}$-reflection of $\Sigma~\!\!\mathbb{J}$ does not exist.

By (i)(b)(c) and Corollary \ref{key corollary 1}, the $\mathbf{K}$-reflection of $\Sigma~\!\!\mathbb{J}$ does not exist, namely, neither the $\mathbf{DC}$-reflection of $\Sigma~\!\!\mathbb{J}$ nor the $\mathbf{QSob}$-reflection of $\Sigma~\!\!\mathbb{J}$ exists.

\item The sobrification of $\Sigma~\!\!\mathbb{J}$ and the $\mathbf{WF}$-reflection of $\Sigma~\!\!\mathbb{J}$ coincide.

By (i)(iii) and Corollary \ref{sobrification agree K-reflection 4}, the sobrification of $\Sigma~\!\!\mathbb{J}$ and the $\mathbf{WF}$-reflection of $\Sigma~\!\!\mathbb{J}$ coincide, and $\langle (\Sigma~\!\!\mathbb{J})_{\top}^{\natural},\eta_{\Sigma~\!\!\mathbb{J}}^{\natural} \rangle$ is both a sobrification of $\Sigma~\!\!\mathbb{J}$ and a $\mathbf{WF}$-reflection of $\Sigma~\!\!\mathbb{J}$.

 \end{enumerate}
\end{example}

Finally, we present a $T_0$ space $Z$ such that neither the $\mathbf{RD}$-reflection of $Z$ nor the $\mathbf{WD}$-reflection of $Z$ exists, and also neither the $\mathbf{QSob}$-reflection of $X_{coc}$ nor the $\mathbf{DC}$-reflection of $X_{coc}$ exists.

\begin{example}\label{example:3.9} Suppose that $\mathbf{K}\in \{\mathbf{Top}_{rd}, \mathbf{Top}_{wd}, \mathbf{QSob}, \mathbf{Top}_{dc}\}$. Let $X$ be a uncountably infinite set and $X_{coc}$ the space equipped with \emph{the co-countable topology} (the empty set and the complements of countable subsets of $X$ are open). Then
\begin{enumerate}[\rm (a)]
    \item $\mathcal C(X_{coc})=\{\emptyset, X\}\cup X^{(\leqslant\omega)}$, $X_{coc}$ is $T_1$ and hence a $d$-space, and the specialization order on $X_{coc}$ is the discrete order.
    \item $\ir_c(X_{coc})=\{\overline{\{x\}} : x\in X\}\bigcup\{X\}=\{\{x\} : x\in X\}\bigcup\{X\}$ and hence $X_{coc}$ is not sober.
    \item  $X_{coc}$ is both a weakly sober space and a cut space, but it is not a quasisober space.

    Clearly, $\{x\}^{\delta}=\{x\}$ for each $x\in X$ and $X^{\delta}=X$. Whence $X_{coc}$ is a weakly sober space and hence a cut space. Since the specialization order on $X_{coc}$ is the discrete order, $\mathcal D(X_{coc})=\{\{x\} : x\in X\}$, and consequently, $X\neq D^{\delta}$ for any $D\in \mathcal D(X_{coc})$. So $X_{coc}$ is not a quasisober space.

    \item $\mk (X_{coc})=X^{(<\omega)}\setminus \{\emptyset\}$ and $X_{coc}$ is well-filtered.

    Clearly, every finite subset is compact. Conversely, if $C\subseteq X$ is infinite, then $C$ has an infinite countable subset $\{c_n : n\in\mn\}$. Let $C_0=\{c_n : n\in\mn\}$ and $U_m=(X\setminus C_0)\cup \{c_m\}$ for each $m\in \mn$. Then $\{U_n : n\in\mn\}$ is an open cover of $C$, but has no finite subcover. Whence $C$ is not compact. Thus $\mk (X_{coc})=X^{(<\omega)}\setminus \{\emptyset\}$.

    Now we show that $X_{coc}$ is well-filtered. Suppose that $\{F_d : d\in D\}\subseteq \mk (X_{coc})$ is a filtered family and $U\in \mathcal O(X_{coc})$ with $\bigcap_{d\in D}F_d\subseteq U$. As $\{F_d : d\in D\}$ is filtered and all $F_d$ are finite, $\{F_d : d\in D\}$ has a least element $F_{d_0}$, and hence $F_{d_0}=\bigcap_{d\in D}F_d\subseteq U$, proving that $X_{coc}$ is well-filtered.

   \item  $X_{coc}$ is not a $\mathsf{WD}$ space and hence it is neither a Rudin space nor a $\mathsf{DC}$ space.

    By (b)(d) and Theorem \ref{soberequiv}, $X_{coc}$ is not a $\mathsf{WD}$ space and hence it is neither a Rudin space nor a $\mathsf{DC}$ space by Corollary \ref{SDRWspacerelation}.

\item $X_{coc}$ is not a $\mathbf{K}$ space and $\langle ({X_{coc}})_{\top}^{\flat},\eta_{X_{coc}}^{\flat} \rangle$ is a sobrification of $X_{coc}$.

  By (b)(c)(e) and Corollary \ref{key corollary 1}, $X_{coc}$ is not a $\mathbf{K}$ space and $\langle ({X_{coc}})_{\top}^{\flat},\eta_{X_{coc}}^{\flat} \rangle$ is a sobrification of $X_{coc}$.

  \item The sobrification $\langle ({X_{coc}})_{\top}^{\flat},\eta_{X_{coc}}^{\flat} \rangle$ of $X_{coc}$ is not a $\mathbf{K}$-reflection of $X_{coc}$.

  Assume, on the contrary, that $\langle ({X_{coc}})_{\top}^{\flat},\eta_{X_{coc}}^{\flat} \rangle$ is a $\mathbf{K}$-reflection of $X_{coc}$. Let $Y=X\cup \mathbb{N}$ ($\mathbb{N}$ is the set of natural number with the usual order) with ordering defined by $x\leq_Y y$ if{}f $x=y$ or $x, y\in \mathbb{N}$ and $x \leq_{\mathbb{N}} y$. Endow $Y$ with the upper topology $\upsilon(Y)$. Then
  \begin{enumerate}[\rm (i)]

  \item $Y$ is not a dcpo ($\mathbb{N}$, as a chain, has not a
greatest lower bound in $Y$) and hence $(Y, \upsilon(Y))$ is not a $d$-space. So $(Y, \upsilon(Y))$ is nether a well-filtered space and nor a sober space.

  \item $\mathcal C((Y, \upsilon(Y)))=\{Y, \emptyset\}\cup X^{(<\omega)}\cup \{F\cup \downarrow n: F\in X^{(<\omega)}, n\in \mathbb{N}\}$ (note that $\emptyset\in X^{(<\omega)}$ and so $\{\downarrow n: n\in \mathbb{N}\}\subseteq C((Y, \upsilon(Y)))$). Whence, $\ir_c((Y,\upsilon (Y)))=\{\cl_{\upsilon(Y)}\{y\}: y\in Y\}\cup \{Y\}=\{\downarrow n: n\in \mathbb{N}\}\cup\{\{x\}: x\in X\}\cup \{Y\}$.

       \item $\mathsf{K}((Y, \upsilon (Y)))=\{\ua G : \emptyset\neq G\subseteq Y\}$ and $(Y, \upsilon(Y))$ is not well-filtered.

       By $\mathcal C((Y, \upsilon(Y)))=\{Y, \emptyset\}\cup X^{(<\omega)}\cup \{F\cup \downarrow n: F\in X^{(<\omega)}, n\in \mathbb{N}\}$, it is easy to verify that every subset of $Y$ is compact in $(Y, \upsilon(Y))$. Whence $\mathsf{K}((Y, \upsilon (Y)))=\{\ua G : \emptyset\neq G\subseteq Y\}$. Clearly, $\mathcal K_{\mathbb{N}}=\{\ua n : n\in \mathbb{N}\}\subseteq \mathsf{K}((Y, \upsilon (Y)))$ is a filtered family and $\bigcap K_{\mathbb{N}}=\emptyset$, but $\ua n\neq \emptyset$ for any $n\in\mathbb{N}$. Thus $(Y, \upsilon(Y))$ is not well-filtered.

\item $(Y,\upsilon (Y))$ is a quasisober space.

 Clearly, $\mathbb{N}$ is a chain in $(Y, \upsilon(Y))$ (its specialization order is the original order of $Y$) and $\mathbb{N}^{\delta}=(\mathbb{N}^{\ua})^{\da}=\emptyset^{\da}=Y$. Whence by $\ir_c((Y,\upsilon (Y)))=\{\cl_{\upsilon(Y)}\{y\}: y\in Y\}\cup \{Y\}$, $(Y,\upsilon (Y))$ is a quasisober space.

       \item $\mathcal D_c((Y,\upsilon (Y))=\mathsf{RD}((Y,\upsilon (Y))=\mathsf{WD}((Y,\upsilon (Y))=\ir_c((Y,\upsilon (Y)))$. So $(Y,\upsilon (Y))$ is a $\mathsf{DC}$ space and hence it is both a Rudin space and a $\mathsf{WD}$ space.

        Since $\mathbb{N}$ is a chain and $Y=\cl_{\upsilon (Y)} \mathbb{N}$, $Y\in \mathcal D_c((Y,\upsilon (Y))$. By $\ir_c((Y,\upsilon (Y)))=\{\cl_{\upsilon(Y)}\{y\}: y\in Y\}\cup \{Y\}$ and Lemma \ref{DRWIsetrelation}, we have that $\mathcal D_c((Y,\upsilon (Y))=\mathsf{RD}((Y,\upsilon (Y))=\mathsf{WD}((Y,\upsilon (Y))=\ir_c((Y,\upsilon (Y)))$. So $(Y,\upsilon (Y))$ is a $\mathsf{DC}$ space, whence it is both a Rudin space and a $\mathsf{WD}$ space.
\item $(Y,\upsilon (Y))$ is a $\mathbf{K}$-space (by (iv)(v)).
\end{enumerate}

        Define the mapping $i_{X_{coc}} : X_{coc}\rightarrow (Y,\upsilon (Y))$ by $i_{X_{coc}}(x)=x$ for each $x\in X$. As $U\cup \mathbb{N}\in \upsilon (Y)$ and $V\cap X\in \mathcal{O}(X_{coc})$ for any $(U, V)\in \mathcal O(X_{coc})\times\upsilon (Y)$, $i_{X_{coc}}$ is a topological embedding.
        Since $\langle ({X_{coc}})_{\top}^{\flat},\eta_{X_{coc}}^{\flat} \rangle$ is a $\mathbf{K}$-reflection of $X_{coc}$, there exists a unique continuous mapping $(i_{X_{coc}})^{k}: ({X_{coc}})_{\top}^{\flat}\rightarrow (Y, \upsilon (Y))$ such that $(i_{X_{coc}})^{k}\circ \eta_{X_{coc}}^{\flat}=i_{X_{coc}}$, that is, the following diagram commutes.

\begin{equation*}
\centerline{
\xymatrix{ X_{coc} \ar[dr]_{i_{X_{coc}}} \ar[r]^-{\eta_{X_{coc}}^{\flat}}&  ({X_{coc}})_{\top}^{\flat}\ar@{.>}[d]^{(i_{X_{coc}})^{k}} & \\
  & (Y, \upsilon (Y))  & &
   }}
\end{equation*}

Then $(i_{X_{coc}})^{k}(x)=(i_{X_{coc}})^{k}(\eta_{X_{coc}}^{\flat}(x))=i_{X_{coc}}(x)=x$ for each $x\in X$ and $(i_{X_{coc}})^{k}(\top)\in Y$. For each $x\in X$, since $x\leq \top$, we have $(i_{X_{coc}})^{k}(x)=x \leq (i_{X_{coc}})^{k}(\top)$. Whence $(i_{X_{coc}})^{k}(\top)$ is an upper bound of $X$ in $Y$, a contradiction.

Thus $\langle ({X_{coc}})_{\top}^{\flat},\eta_{X_{coc}}^{\flat} \rangle$ is not a $\mathbf{K}$-reflection of $X_{coc}$.

\item The $\mathbf{K}$-reflection of $X_{coc}$ does not exist.

  By (b)(f)(g) and Corollary \ref{key corollary 1}, the $\mathbf{K}$-reflection of $X_{coc}$ does not exist, namely, neither the $\mathbf{RD}$-reflection of $X_{coc}$ nor the $\mathbf{WD}$-reflection of $X_{coc}$ exists, and also neither the $\mathbf{QSob}$-reflection of $X_{coc}$ nor the $\mathbf{DC}$-reflection of $X_{coc}$ exists.

\end{enumerate}
\end{example}

By the above four examples, we get the main result of this paper.

\begin{theorem}
For $\mathbf{K}\in \{\mathbf{Top}_{dc}, \mathbf{Top}_{rd}, \mathbf{Top}_{wd}, \mathbf{QSob}, \mathbf{WSob}, \mathbf{Top}_{cut}\}$, $\mathbf{K}$ is not reflective in $\mathbf{Top}_0$.

\end{theorem}

\end{document}